\documentclass[12pt,reqno]{amsart}

\usepackage[english]{babel}

\usepackage{amsthm}
\usepackage{amssymb}
\usepackage{amsmath}
\usepackage{amsaddr} 
\usepackage{comment}
\usepackage{xcolor}
\usepackage{hyperref}

\theoremstyle{definition}
\newtheorem{defi}{Definition}
\newtheorem{example}[defi]{Example}
\newtheorem{remark}[defi]{Remark}
\theoremstyle{plain}
\newtheorem{theorem}[defi]{Theorem}
\newtheorem{lem}[defi]{Lemma}

\newtheorem{prop}[defi]{Proposition}
\newtheorem{cor}[defi]{Corollary}

\def\IN{\mathbb{N}} 
\def\IZ{\mathbb{Z}} 
\def\IQ{\mathbb{Q}} 
\def\IR{\mathbb{R}} 

\def\id{\mathrm{id}} 
\def\dom{\mathrm{dom}} 

\def\leqS{\leq_\mathrm{S}}

\usepackage{soul}

\begin{document}

\title[]{Regainingly approximable numbers\\ and sets} 

\author{Peter Hertling, Rupert H\"olzl and Philip Janicki}
\address{Fakult\"at f\"ur Informatik, Universit\"at der Bundeswehr M\"unchen, \\ 
85577 Neubiberg, Germany}
\email{peter.hertling@unibw.de}
\email{r@hoelzl.fr}
\email{philip.janicki@unibw.de}


\begin{abstract} We call an $\alpha \in \IR$ \emph{regainingly approximable} if there exists
a computable nondecreasing sequence $(a_n)_n$ of rational numbers converging to $\alpha$
with $\alpha - a_n < 2^{-n}$ for infinitely many~${n \in \IN}$.
We also call a set $A\subseteq\IN$ \emph{regainingly approximable} if it is c.e.~and the strongly left-computable
number $2^{-A}$ is regainingly approximable.
We show that the set of regainingly approximable sets is neither closed under union nor intersection and that every c.e.~Turing degree contains such a set.
Furthermore, the regainingly approximable numbers lie properly between the computable and the left-computable numbers and are not closed under addition.
While regainingly approximable numbers are easily seen to be i.o.\  $K$\nobreakdash-trivial, we construct such an~$\alpha$ such that ${K(\alpha \restriction n)>n}$ for infinitely many~$n$. Similarly, there exist regainingly approximable sets whose initial segment complexity infinitely often reaches the maximum  possible for c.e.\ sets. Finally, there is a uniform algorithm splitting regular real numbers into two regainingly approximable numbers that are still regular.

\smallskip
\noindent
{\bf Keywords:} left-computable numbers; effective approximation; computably enumerable sets; splitting; Turing degrees; Kolmogorov complexity

\smallskip
\noindent
{\bf AMS classification:} 03D78, 03D25, 03D30, 03D32, 68Q30
\end{abstract}

\maketitle


\section{Introduction}

We call a sequence $(a_n)_n$ of real numbers
\emph{increasing} if, for all $n\in\IN$, $a_n < a_{n+1}$, and 
\emph{nondecreasing} if, for all $n\in\IN$, $a_n \leq a_{n+1}$, and we define the terms
\emph{decreasing} and \emph{nonincreasing} analoguosly.
A real number is called \emph{left-computable}
if there exists a computable nondecreasing sequence of rational numbers converging to it;
note that some textbooks~\cite{Nie2009,DH2010} call such numbers \emph{left-c.e.}.
A real number $\alpha$ is called \emph{computable}
if there exists a computable sequence $(a_n)_n$ of rational numbers 
satisfying $|\alpha - a_n| < 2^{-n}$, for all $n\in\IN$.
It is easy to see that any computable real number is left-computable.
In this article, we study real numbers that are limits of computable, nondecreasing, converging sequences $(a_n)_n$ of rational
numbers which are not required to satisfy the condition $|\alpha - a_n| < 2^{-n}$ for all but only for infinitely many~$n\in\IN$.

\begin{defi}
\label{definition:regaining}
We call a real number $\alpha$ \emph{regainingly approximable} if there exists
a computable nondecreasing sequence of rational numbers~$(a_n)_n$ converging to $\alpha$
such that ${\alpha - a_n < 2^{-n}}$ holds for infinitely many~${n \in \IN}$.
\end{defi}
Intuitively speaking, regainingly approximable numbers are not required to possess  approximations that converge as speedily as those to computable numbers, but they must possess approximations that ``catch up'' with that speed infinitely often  while being allowed to ``dawdle'' arbitrarily in between.
Trivially, every regainingly approximable number is left-computable,
and every computable number is regainingly approximable.
In fact, even the following stronger observation holds.
\begin{example}
\label{example:non-high_numbers}
Every non-high left-computable number is regainingly approximable.
Indeed, let $\alpha$ be a left-computable number that is not high.
Fix a computable nondecreasing sequence of rational numbers $(a_n)_n$ converging to~$\alpha$. An increasing function $s \colon \IN \to \IN$
with ${\alpha - a_{s(n)} < 2^{-n}}$ for all $n \in \IN$ can be computed from
oracle $\alpha$. 
Since $\alpha$~is not high, according to a characterization of highness by Martin~\cite{Mar1966} (see Soare~\cite[Theorem XI.1.3]{Soa1987}),
there exists a computable, w.l.o.g.\ increasing function $r \colon \IN\to\IN$
with $r(n) \geq s(n)$ for infinitely many~${n \in \IN}$.
Then for any $n \in \IN$ with $r(n) \geq s(n)$ we obtain
\begin{equation*}
	\alpha - a_{r(n)} \leq \alpha - a_{s(n)} < 2^{-n} .
\end{equation*}
Therefore, $\alpha$ is regainingly approximable.
\end{example}

The remainder of the article is structured as follows:
\begin{itemize}
\item
In Section~\ref{section:characterizations} we begin by showing that Definition~\ref{definition:regaining}	is robust under minor modifications, with the equivalences between the different formulations turning out to be effectively uniform. We also state some computability-theoretic properties of the set $\{n \in \IN \colon \alpha - a_n < 2^{-n}\}$, that is, the set of points where a regaining approximation $(a_n)_n$ ``catches up.''
	
\item
	We begin Section~\ref{section:sets} by giving several characterizations of the sets $A\subseteq \IN$ that have the property that the real number 
	\[2^{-A} := \sum_{a \in A} 2^{-(a+1)} \]
	is regainingly approximable.
	For the rest of the section we then focus on c.e.~sets $A\subseteq\IN$: 
	A real number $x\in [0,1]$ is called \emph{strongly left-computable} if there exists a computably enumerable set
	$A\subseteq \IN$ with $x=2^{-A}$.
	It is well-known that the set of strongly left-computable numbers is a proper superset of the set of computable
   numbers in $[0,1]$ and a proper subset of the set of left-computable numbers in $[0,1]$.
	We consider c.e.~sets~$A$ such that $2^{-A}$ is regainingly approximable 
   and call such sets \emph{regainingly approximable}. We give different characterizations of this class of sets, and note that, unlike for regainingly approximable numbers, not all arguments here can be 
   fully effectively uniform.
		
\item
	In Section~\ref{section:sets-computability},
	we state further computability-theoretic properties of regainingly approximable sets.
	First, we observe an easy splitting result, namely that every c.e.~set $C\subseteq\IN$ is the union of two disjoint regainingly approximable sets $A,B\subseteq\IN$. Next, we prove that there is a c.e.~set that is not regainingly approximable.
	Finally, we show that every c.e.~Turing degree contains a regainingly approximable set.
	
\item
	In Sections~\ref{section:Kolmogorov-complexity} and~\ref{sdfhsdfhjsdfdgdgddf} we look at the Kolmogorov complexity of regainingly approximable
	numbers and sets.
	On the one hand, we observe that every regainingly approximable number is i.o.\  $K$-trivial and, hence, not 
	Martin-L\"of random.
On the other hand, we show that there exists a regainingly approximable number  such that the prefix-free Kolmogorov complexity of infinitely many of its initial segments (in a sense to be explained in Section~\ref{dfnasdnsdfdfg}) exceeds their length, and that there exist regainingly approximable sets whose initial segments infinitely often have the maximal  Kolmogorov complexity that is possible for a c.e.\ set.
	
\item
	In Section~\ref{section:arithmetical-properties} we observe that regainingly approximable sets
	and numbers behave badly with respect to arithmetical operations:
	The set of regainingly approximable sets is closed neither under union nor under intersection, and while the set of regainingly approximable numbers is closed downwards under Solovay reducibility it is not closed under addition.
	
\item
	Finally, in Section~\ref{section:SplittingRegularReals} 
	we formulate a uniformly effective  algorithm for splitting c.e.~sets into two regainingly approximable sets, which strengthens the splitting result of Section~\ref{section:sets-computability}. In fact, this algorithm proves a stronger result, namely that every real that is regular in the sense of Wu~\cite {Wu2005} can be split 
	into two regular reals that are additionally regainingly approximable.
\end{itemize}
We close this introduction by mentioning that isolating the notion of regaining approximability was a key step needed for H\"olzl and Janicki's~\cite{HJ2023b} negative answer to the question posed by Merkle and Titov~\cite{MT2020} whether among the left-computable numbers, being Martin-L\"of random is equivalent to being non-speedable.

\section{Notation}\label{dfnasdnsdfdfg}

For general background on computability theory and algorithmic randomness, we refer the reader to the usual textbooks~\cite{Soa1987,Nie2009,DH2010}. 
We will use \emph{Cantor's pairing function} $\langle \cdot,\cdot \rangle\colon \IN^2\to\IN$
defined by
\[ \langle m,n\rangle := \frac{1}{2}\left(m + n\right)\left(m + n + 1\right) + n , \]
for all~$m,n\in\IN$. This is a computable bijection, so let
$\pi_1\colon \IN\to\IN$ and $\pi_2\colon \IN\to\IN$ denote the two components of its inverse function, that is,  
$\langle \pi_1(n),\pi_2(n)\rangle = n$ for all~${n\in\IN}$. 
For $n\in\IN$, we write $\log(n)$ for the length of the binary representation of $n$, thus
$\log(n)=\lfloor \log_2(n) \rfloor + 1$ if $n>0$ and $\log_2(0)=1$. Hence, $\log(n)$ is  identical to the usual binary logarithm up to a small additive constant.

For functions $f,g\colon \IN \to \IN$ we write $f \leq^+ g$ if there exists a constant~${c \in \IN}$ such that $f(n) \leq g(n) + c$ for all~$n$.
A set $A\subseteq\IN$ will often be identified with its characteristic sequence, that is, we define $A(n):=1$ if $n\in A$, and $A(n):=0$ if $n\not\in A$.

\medskip

We end this section by laying out the following conventions about how, in the remainder of the article, we will (often tacitly) identify real numbers with infinite binary sequences and subsets of the natural numbers:
\begin{itemize}
	\item If we ascribe to some $x\in \IR$ a property that is formally defined only for infinite binary sequences, then we mean that a specific, 
    uniquely determined infinite binary sequence has that property; namely, 
	the sequence containing infinitely many ones and that is a binary representation of the unique real number $y\in(0,1]$ with $x-y\in\IZ$.
	\item If we ascribe  to some $x\in \IR$  a property that is formally defined only for subsets of $\IN$, then we mean that the uniquely determined infinite set $A\subseteq\IN$, whose characteristic sequence
	is the infinite binary sequence defined in the previous item, has that property.
	\item When we talk about ``initial segments'' of some $x\in \IR$, then we are again referring to the 
	initial segments of the infinite binary sequence defined in the first item.	This will be used for instance when we talk about the plain or prefix-free Kolmogorov complexity of initial segments of some such~$x$.
	\item Similarly, when we talk about the initial segments of an $A\subseteq\IN$, then we are referring to its characteristic sequence.
\end{itemize}

\section{Robustness}
\label{section:characterizations}

In this section, we first show that slight changes to the definition of regainingly approximable numbers do not lead to a different notion.
%
The following lemma will be useful; note that no computability assumptions are made.

\begin{lem}
\label{lemma:unbounded-function}
Let $(a_n)_n$ be a nondecreasing sequence of real numbers converging to some real number $\alpha$
such that, for infinitely many $n \in \IN$, $\alpha - a_n < 2^{-n}$.
Then, for every unbounded function $f\colon\IN\to\IN$ there exist infinitely many $m$ with
$\alpha - a_{f(m+1)} < 2^{-f(m)}$.
\end{lem}

\begin{proof}
By assumption, the set
\[ A := \{n\in\IN \colon n  \geq f(0)  \text{ and }\alpha - a_n < 2^{-n}\} \]
is infinite.
We define a function $g\colon A\to\IN$ by 
\[ g(n):=\min\{m\in\IN \colon f(m+1) > n \} , \]
for $n\in A$. The function $g$ is well-defined because $f$ is unbounded.
For every $n\in A$ we have $f(g(n)) \leq n < f(g(n)+1)$.
The set 
\[g(A) := \{g(n) \colon n \in A\}\]
is infinite.
Let $m\in g(A)$ be arbitrary
and let $n\in A$ be such that $m=g(n)$. Then
\[ \alpha - a_{f(m+1)} = \alpha - a_{f(g(n)+1)} \leq \alpha - a_n < 2^{-n} \leq 2^{-f(g(n))} = 2^{-f(m)} . 
\qedhere 
\]
\end{proof}

There are several obvious modifications of Definition~\ref{definition:regaining}
that one might want to consider. 
First, instead of allowing nondecreasing~$(a_n)_n$, we might require $(a_n)_n$ to be increasing.
Secondly, one might replace the condition $\alpha - a_n < 2^{-n}$ by 
the condition ${\alpha - a_n < 2^{-f(n)}}$ where ${f\colon \IN\to\IN}$
is an arbitrary computable, unbounded function of one's choice;
or, one might ask for this to hold only for {\em some} fixed computable, nondecreasing, unbounded function $f\colon \IN\to\IN$, a seemingly weaker requirement.
However, it turns out that none of these modifications make any difference.

\begin{prop}
\label{prop:equivalent-conditions}
For a real number $\alpha\in\IR$ the following statements are equivalent:
\begin{enumerate}
\item
$\alpha$ is regainingly approximable.
\item
There exists a computable, increasing sequence of rational numbers $(a_n)_n$ converging to $\alpha$
such that, for infinitely many $n \in \IN$, $\alpha - a_n < 2^{-n}$.
\item
For every computable, unbounded function $f\colon \IN\to\IN$ 
there exists a computable increasing sequence of rational numbers $(a_n)_n$ converging to $\alpha$
such that, for infinitely many $n \in \IN$, \[\alpha - a_n < 2^{-f(n)}.\]
\item
There exist a computable, nondecreasing, and unbounded function ${f\colon \IN\to\IN}$ and
a computable nondecreasing sequence of rational numbers $(a_n)_n$ converging to $\alpha$
such that, for infinitely many $n \in \IN$, $\alpha - a_n < 2^{-f(n)}$.
\end{enumerate}
\end{prop}

Note that this implies in particular that it makes no difference whether we use ``$<$'' or ``$\leq$''
in the definition of regaining approximability. We would also like to point out that
all implications in the following proof are uniformly effective.

\begin{proof}\,

\smallskip	
	
$(2) \Rightarrow (1)$: Trivial.

\smallskip

$(3) \Rightarrow (2)$: Trivial.

\smallskip

$(1) \Rightarrow (3)$:
Let $\alpha$ be a regainingly approximable number.
Let $(b_n)_n$ be a computable nondecreasing sequence of rational numbers converging to $\alpha$
with $\alpha - b_n < 2^{-n}$ for infinitely many $n \in \IN$.
Let $f\colon \IN\to\IN$ be a computable, unbounded function.
Then the function $g\colon \IN\to\IN$ defined by 
\[ g(n) := 1+n+\max \{ f(m) \colon m \leq n \} \]
is computable, increasing, and satisfies $g(n)\geq f(n)+1$, for all $n\in\IN$.
In particular, $g$ is unbounded.
The sequence $(a_n)_n$ of rational numbers defined by
\[ a_n:=b_{g(n+1)} - 2^{-g(n)}\]
is computable and increasing and converges to $\alpha$.
By Lemma~\ref{lemma:unbounded-function}
there exist infinitely many $n$ with $\alpha - b_{g(n+1)} < 2^{-g(n)}$.
For all such $n$ we obtain
\[ \alpha - a_n = \alpha - b_{g(n+1)} + 2^{-g(n)} < 2^{-g(n)+1} \leq  2^{-f(n)} . \]
  
  \smallskip
  
$(1) \Rightarrow (4)$: Trivial.

\smallskip

$(4) \Rightarrow (1)$:
Assume that $f\colon \IN\to\IN$ is a computable, nondecreasing, and unbounded function 
and $(b_n)_n$ is a computable nondecreasing sequence of rational numbers converging to $\alpha$
such that, for infinitely many $n \in \IN$, $\alpha - b_n < 2^{-f(n)}$.
Define a function $g\colon \IN\to\IN$ via
\[ g(0):= \max\{m\in\IN \colon f(m)=f(0)\}
\]
and
\[ g(n+1) := \max\{m\in\IN \colon f(m)=f(g(n)+1)\} ,
\]
for $n\in\IN$; informally speaking, the graph of $f$ is an infinite increasing sequence of plateaus of finite length
and $g(n)$ gives the $x$-coordinate of the right-most point in the $(n+1)$-th plateau.

Clearly, $g$~is computable and increasing and satisfies $f(g(n))\geq n$
for all $n\in\IN$.
Furthermore, for every $k\in\IN$ there exists exactly
one $n\in\IN$ with $f(k)=f(g(n))$, and this $n$ satisfies $k \leq g(n)$.
The sequence $(a_n)_n$ of rational numbers defined by
$a_n := b_{g(n)}$, 
for all $n\in\IN$, is computable and nondecreasing and converges to $\alpha$.
By assumption, the set 
\[ B := \{k\in\IN \colon \alpha - b_k < 2^{-f(k)} \} \] 
is infinite.
Hence, the set
\[ A := \{n \in \IN \colon (\exists k \in B)\; f(k) = f(g(n)) \} \]
is infinite as well. 
Consider a number $n\in A$,
and let $k\in B$ be a number with $f(k)=f(g(n))$. Then $k\leq g(n)$ and
\[ \alpha - a_n = \alpha - b_{g(n)} \leq \alpha - b_k < 2^{-f(k)} = 2^{-f(g(n))} \leq 2^{-n} . 
\qedhere
\]
\end{proof}

Next, we observe that if a left-computable number $\alpha$ is regainingly approximable
then this will be apparent no matter which of its effective approximations we look at.

\begin{prop}
\label{prop:index-function}
Let $\alpha$ be a left-computable number, and let $(a_n)_n$ be a computable, nondecreasing
sequence of rational numbers converging to $\alpha$.
Then the following conditions are equivalent.
\begin{enumerate}
\item
$\alpha$ is a regainingly approximable number.
\item
There exists a computable, increasing function $r\colon \IN\to\IN$ such that, for 
infinitely many $n$, $\alpha - a_{r(n)} < 2^{-n}$.
\end{enumerate}
\end{prop}

Note that the proof is effectively uniform in both directions.

\begin{proof}
$(2) \Rightarrow (1)$: Assume that there exists a computable, increasing function
$r\colon \IN\to\IN$ such that  we have ${\alpha - a_{r(n)} < 2^{-n}}$ for infinitely many $n$.
Then the sequence $(b_n)_n$ of rational numbers defined by
${b_n:=a_{r(n)}}$ is computable, nondecreasing, converges to $\alpha$, and satisfies, for infinitely
many $n$, $\alpha - b_n < 2^{-n}$. Hence, $\alpha$ is regainingly approximable.

\smallskip

$(1) \Rightarrow (2)$: Assume that $\alpha$ is regainingly approximable.
By Proposition~\ref{prop:equivalent-conditions}
there exists a computable, increasing sequence $(b_n)_n$ of rational numbers converging to $\alpha$
such that there exist infinitely many~$n$ with $\alpha - b_n < 2^{-n}$.
We define a computable, increasing function ${r\colon \IN\to\IN}$ by
$r(0):=\min\{m\in\IN \colon a_m\geq b_0 \}$, and
\[ r(n+1) := \min\{m\in\IN \colon m > r(n) \text{ and } a_m\geq b_{n+1} \} , \]
for $n\in\IN$.
For all $n\in\IN$ we have $a_{r(n)} \geq b_n$.
Thus, for the infinitely many $n\in\IN$ with $\alpha - b_n < 2^{-n}$,
we obtain
\[ \alpha - a_{r(n)} \leq \alpha - b_n < 2^{-n} . 
\qedhere
\]
\end{proof}
	

We close this section by investigating the computability-theoretic properties of the set of points where a regaining approximation  ``catches up.''
\begin{prop}
Let $\alpha$ be a regainingly approximable number,
and let $(a_n)_n$ be a computable nondecreasing sequence of rational numbers
converging to $\alpha$ with $\alpha - a_n < 2^{-n}$ for infinitely many $n \in \IN$.
Then
\begin{equation*}
	A:= \{n \in \IN \colon \alpha - a_n < 2^{-n}\}
\end{equation*}
has the following properties:
\begin{enumerate}
\item
$\IN \setminus A$ is computably enumerable.
\item
The following are equivalent:
\begin{enumerate}
	\item $\alpha$ is computable.
	\item $A$ is decidable.
	\item $A$ is not hyperimmune.
\end{enumerate}
\end{enumerate}
\end{prop}

\begin{proof}\,
\begin{enumerate}
\item
If $\alpha$ is rational or if $\IN\setminus A$ is empty then $A$ and $\IN\setminus A$ are decidable; thus assume w.l.o.g.\ that $\alpha$ is irrational and that $\IN \setminus A$ is non-empty.
Then 
$\alpha - a_n\geq 2^{-n}$ if and only if $\alpha - a_n > 2^{-n}$.
Fix any~${e \in \IN \setminus A}$, and define a computable function $f \colon \IN \to \IN$ via
\begin{equation*}
	f(\left\langle m,n\right\rangle ) := \begin{cases}
		n	&\text{if $a_m - a_n \geq 2^{-n}$,} \\
		e	&\text{otherwise.}
	\end{cases}
\end{equation*}
Clearly, $f$ computably enumerates $\IN \setminus A$.
\item
(a) $\Rightarrow$ (b):
Suppose that $\alpha$ is computable.
Then there exists a computable sequence $(b_m)_m$ of rational numbers
with $\left|\alpha - b_m\right| < 2^{-m}$ for all $m \in \IN$.
Due to~(1) it suffices to show that $A$ is computably enumerable.
Fix some $d \in A$ and define a computable function $g \colon \IN \to \IN$ via
	\begin{equation*}
		g(\left\langle m,n\right\rangle ) := \begin{cases}
			n	&\text{if $b_m - a_n + 2^{-m} < 2^{-n}$,} \\
			d	&\text{otherwise.}
		\end{cases}
	\end{equation*}
Then $g$ computably enumerates $A$; indeed, if for some $m$ it holds that ${b_m - a_n + 2^{-m} < 2^{-n}}$ then ${\alpha - a_n < 2^{-n}}$; and for the other direction, if ${\alpha - a_n < 2^{-n}}$ holds, then 
for all $m$ satisfying ${\alpha - a_n + 2\cdot 2^{-m} < 2^{-n}}$ we have ${b_m - a_n + 2^{-m} < 2^{-n}}$.\\
(b) $\Rightarrow$ (c):
Trivial. \\
(c) $\Rightarrow$ (a):
By assumption $A$ is infinite. Define ${p_A \colon \IN \to \IN}$ as the uniquely determined increasing function
with ${p_A(\IN) = A}$.
Suppose that $A$ is not hyperimmune.
Then there exists a computable function $r \colon \IN \to \IN$ with $r(n) \geq p_A(n)$ for all $n \in \IN$ (see, for instance, Soare~\cite[Theorem V.2.3]{Soa1987}). Then we have
\begin{equation*}
	\alpha - a_{r(n)} \leq \alpha - a_{p_A(n)} < 2^{-p_A(n)} \leq 2^{-n}
\end{equation*}
for all $n \in \IN$. Hence, $\alpha$ is computable.
\qedhere
\end{enumerate}
\end{proof}

\section{Regainingly Approximable Sets}
\label{section:sets}

In this section we study sets $A\subseteq\IN$ such that 
$2^{-A}$ is regainingly approximable. We start by characterizing these sets in ways analogous to Propositions~\ref{prop:equivalent-conditions} and ~\ref{prop:index-function}, before then focusing on the 
particularly interesting case where $A$ is c.e.
%
%

We need the following terminology and easy lemma:
Following Soare~\cite{Soa1987} we call a sequence $(A_n)_n$ of finite sets $A_n\subseteq\IN$ a \emph{strong array}
if the function $n\mapsto \sum_{i\in A_n} 2^i$ is computable, and a strong array $(A_n)_n$ a
\emph{uniformly computable approximation from the left} of a set $A\subseteq\IN$ if
\begin{enumerate}
\item
for all $n\in\IN$, $A_n\subseteq\{0,\ldots,n-1\}$,
\item 
for all $n\in\IN$, $2^{-A_n} \leq 2^{-A_{n+1}}$, and
\item
for all $i\in\IN$, $\lim_{n\to\infty} A_n(i)=A(i)$.
\end{enumerate}
\begin{lem}[{Soare~\cite{Soa1969}}]
\label{lemma:uniformlycomputableapproximationfromtheleft}
For a set $A\subseteq\IN$ the following conditions are equivalent.
\begin{enumerate}
\item
The number $2^{-A}$ is left-computable.
\item
There exists a uniformly computable approximation from the left of the set $A\subseteq\IN$.
\end{enumerate}
\end{lem}
With this we are ready to establish the following characterization.
\begin{theorem}
\label{theorem:reg-approx-sets-some-approx}
Let $A\subseteq\IN$ be a set such that the number $2^{-A}$ is left-computable.
Then the following conditions are equivalent.
\begin{enumerate}
\item
The number $2^{-A}$ is regainingly approximable.
\item
There exists a uniformly computable approximation $(A_n)_n$ of~$A$ from the left 
such that, for infinitely many $n$, 
\[A\cap\{0,\ldots,n-1\} = A_n.\]
\item
For every uniformly computable approximation $(B_n)_n$ of~$A$ from the left 
there exists a computable, increasing function ${s\colon\IN\to\IN}$ such that,
for infinitely many $n$, 
\[A\cap\{0,\ldots,n-1\} = B_{s(n)}\cap\{0,\ldots,n-1\}.\]
\end{enumerate}
\end{theorem}

\begin{proof}
\,

\smallskip

$(3) \Rightarrow (2)$:
By Lemma~\ref{lemma:uniformlycomputableapproximationfromtheleft} some uniformly computable approximation $(B_n)_n$ of~$A$ from the left exists. Then by assumption a function~$s$ as in~(3) must exist as well. Thus, (2) follows immediately by letting 
$A_n:=B_{s(n)} \cap\{0,\ldots,n-1\}$ for all $n\in \IN$.

\smallskip

$(2) \Rightarrow (1)$:
Let $(A_n)_n$ be a uniformly computable approximation of~$A$
from the left 
such that ${A\cap\{0,\ldots,n-1\} = A_n}$ for infinitely many~$n$.
The sequence $(2^{-A_n})_n$ is a nondecreasing computable sequence of rational numbers converging to~$2^{-A}$.
For the infinitely many $n$ with ${A\cap\{0,\ldots,n-1\} = A_n}$
we have $2^{-A} - 2^{-A_n} \leq \sum_{k=n+1}^\infty 2^{-k} = 2^{-n}$.
Thus, $2^{-A}$ is regainingly approximable.

\smallskip

$(1) \Rightarrow (3)$:
If $A$ is cofinite then it is easy to see that for any computable approximation of~$A$ from the left there is a function $s$ as required by~(3); thus  w.l.o.g.\ let $A$ be coinfinite.

Let $(B_n)_n$ be an  arbitrary uniformly computable approximation of~$A$ from the left.
Then $(2^{-B_n})_n$ is a computable, nondecreasing sequence of rational numbers
converging to $2^{-A}$.
By Proposition~\ref{prop:index-function} there exists a computable increasing function
$r\colon\IN\to\IN$ such that, for infinitely many $n$, $2^{-A} - 2^{-B_{r(n)}} < 2^{-n}$.
If there are infinitely many $n$ with 
\[A\cap\{0,\ldots,n-1\} = B_{r(n)}\cap\{0,\ldots,n-1\}\]
then (3)~holds with $s:=r$. Otherwise there is an $N>0$ such that for all $n\geq N$ we have
$A\cap\{0,\ldots,n-1\} \neq B_{r(n)}\cap\{0,\ldots,n-1\}$; that is one of 
\begin{itemize}
	\item[(i)]	$2^{-A\cap\{0,\ldots,n-1\}} + 2^{-n} \leq 2^{-B_{r(n)}\cap\{0,\ldots,n-1\}}$
	\item[(ii)] $2^{-B_{r(n)}\cap\{0,\ldots,n-1\}} + 2^{-n} \leq 2^{-A\cap\{0,\ldots,n-1\}}$
\end{itemize}
must hold. But (i), together with $A$'s coinfiniteness, would imply
\[ 2^{-A} < 2^{-A\cap\{0,\ldots,n-1\}} + 2^{-n} \leq 2^{-B_{r(n)}\cap\{0,\ldots,n-1\}} 
   \leq 2^{-B_{r(n)}} \leq  2^{-A}, \]
a contradiction.
Thus, for $n\geq N$, (ii)~must hold.
Define $s\colon \IN\to\IN$ via
\[ s(n) := \begin{cases}
   n & \hspace{-0.2cm}\text{if } n< N , \\
   \min\left\{k\colon\, \parbox{6.7cm}{\centering $\phantom{\;\;\wedge} k > s(n-1) \;\;\wedge$ \\
   	$2^{-B_k \cap\{0,\ldots,n-1\}} \,{\geq}\,  2^{-B_{r(n)} \cap\{0,\ldots,n-1\}}{+}2^{-n}$}
   \right\} & \hspace{-0.2cm}\text{if } n\geq N ,
   \end{cases}
\]
and note that this is well-defined as $\lim_{n\to\infty} B_n(i)=A(i)$ for all $i\in\IN$. It is clear that $s$ is increasing and computable.
Fix any of the infinitely many $n\geq N$ with $2^{-A} - 2^{-B_{r(n)}} < 2^{-n}$.

We claim that $A\cap\{0,\ldots,n-1\} = B_{s(n)}\cap\{0,\ldots,n-1\}$. For the sake of a contradiction, assume otherwise; then one of 
\begin{itemize}
	\item[(iii)] $2^{-A\cap\{0,\ldots,n-1\}} + 2^{-n} \leq 2^{-B_{s(n)}\cap\{0,\ldots,n-1\}}$
	\item[(iv)] $2^{-B_{s(n)}\cap\{0,\ldots,n-1\}} + 2^{-n} \leq 2^{-A\cap\{0,\ldots,n-1\}}$
\end{itemize}
must hold. From~(iii) we can deduce
\[ 2^{-A} < 2^{-A\cap\{0,\ldots,n-1\}} + 2^{-n} \leq 2^{-B_{s(n)}\cap\{0,\ldots,n-1\}} \leq 2^{-B_{s(n)}} \leq 2^{-A}, \]
a contradiction; from~(iv) we obtain
\[\begin{array}{r@{\;}c@{\;}l}
   2^{-A} &<& 2^{-B_{r(n)}} + 2^{-n} 
     < 2^{-B_{r(n)}\cap\{0,\ldots,n-1\}} + 2^{-n} + 2^{-n} \\
     &\leq & 2^{-B_{s(n)}\cap\{0,\ldots,n-1\}} + 2^{-n}  
     \leq 2^{-A\cap\{0,\ldots,n-1\}}
     \leq  2^{-A} ,
\end{array}\]
another contradiction.
\end{proof}

For the remainder of this section we will focus on
c.e. sets.
\begin{defi}
We call a set $A\subseteq\IN$ \emph{regainingly approximable} if it is computably
enumerable and the real number $2^{-A}$ is regainingly approximable.
\end{defi}
That we resort to the number $2^{-A}$ in this definition concerning sets may seem unnatural and raises 
the question whether regainingly approximable
sets could alternatively be defined using enumerations of their elements. This is indeed possible, as we will show now.

\medskip

We call a total function $f\colon \IN\to\IN$ an \emph{enumeration} of a set $A\subseteq\IN$ if
\[ A = \{n\in\IN \colon (\exists k \in\IN)\; f(k)=n+1\} . \]
If $f(k)=n+1$ then we say that {\em  $f$ enumerates $n$ into $A$ at stage $k$}.
Note that here $f(k)=0$ encodes that  the function~$f$ does not enumerate anything into $A$ at stage $k$.
It is clear that a set $A\subseteq \IN$ is computably enumerable if and only if
there exists a computable enumeration of $A$.
If $f\colon\IN\to\IN$ is an enumeration of a subset of $\IN$ then, for $t\in\IN$, we write
\[ \mathrm{Enum}(f)[t] :=\{n \in\IN \colon (\exists k \in \IN)\; (k<t \text{ and } f(k)=n+1)\} . \]

\begin{defi}
	\label{definition:rgood}
	Let $r\colon\IN\to\IN$ be a nondecreasing, unbounded function.
	We call an enumeration $f\colon\IN\to\IN$ of a set $A\subseteq\IN$ 
	{\em $r$-good} if there exist infinitely many $n$ such that
	\[ \{0,\ldots,n-1\} \cap A \subseteq \mathrm{Enum}(f)[r(n)] . \]
\end{defi}

\begin{remark}
	\label{remark:without_repetitions}
	We call an enumeration of some~$A$ an \emph{enumeration without repetitions} if for every $n\in A$ there exists exactly one $k\in\IN$ with $f(k)=n+1$. From a given enumeration $f$ of~$A$ one can easily compute one without repetitions by defining, for all $k\in\IN$, 
	\[ \widetilde{f}(k):= \begin{cases}
		f(k) & \text{if } f(k)>0 \text{ and } (f(k)-1) \not\in \mathrm{Enum}(f)[k], \\
		0 & \text{otherwise}.
	\end{cases}
	\]
	Note that if $f$ was $r$-good for some nondecreasing, unbounded function~$r$ then $\widetilde{f}$ is $r$-good as well.
\end{remark}

\begin{example}
\label{ex:dec-regaining}
Let $A\subseteq\IN$ be a decidable set.
Then the function ${f\colon \IN\to\IN}$ defined by
$f(n):=n+1$ if $n\in A$, $f(n):=0$ if $n\not\in A$, is a computable and $\id_\IN$-good enumeration without repetitions of $A$.
\end{example}
\begin{theorem}
\label{theorem:regaining-sets}
For a c.e.~$A\subseteq\IN$ the following conditions are equivalent.
\begin{enumerate}
\item
The number $2^{-A}$ is regainingly approximable.
\item
There exists a computable $\id_\IN$-good enumeration of $A$.
\item
There exists a computable, nondecreasing, unbounded function $r\colon\IN\to\IN$
such that there exists a computable $r$-good enumeration of $A$.
\item
For every computable enumeration $f$ of $A$ there exists
a computable, increasing function $r\colon\IN\to\IN$
such that $f$ is $r$-good.
\end{enumerate}
\end{theorem}
In particular, in analogy to Proposition~\ref{prop:index-function}, if a set~$A$ is regainingly approximable
then this will be apparent no matter which of its effective enumerations we look at.
\begin{proof}[{Proof of Theorem~\ref{theorem:regaining-sets}}]
$(3) \Rightarrow (1)$:
Let $r\colon\IN\to\IN$ be a
computable, nondecreasing, unbounded function,
and let $f\colon\IN\to\IN$ be a computable $r$-good enumeration of $A$.
Then $(a_n)_n$ defined for all $n\in\IN$ via
\[ a_n:= 2^{-\mathrm{Enum}(f)[r(n)]}  \]
is a computable, nondecreasing sequence  of rational numbers converging to $2^{-A}$;
and since we have for infinitely many $n$ that
\[ \{0,\ldots,n-1\} \cap A \subseteq \mathrm{Enum}(f)[r(n)] \]
it follows that $2^{-A} - a_n \leq \sum_{k=n}^\infty 2^{-k-1} = 2^{-n}$.
Thus $2^{-A}$ is regainingly approximable.

\smallskip

$(1) \Rightarrow (4)$:
Let $A\subseteq \IN$ be a c.e.~set such that $2^{-A}$ is regainingly approximable,
and let $f\colon\IN\to\IN$ be an arbitrary computable enumeration of~$A$.
Then $(a_n)_n$ defined by
$a_n:=2^{-\mathrm{Enum}(f)[n]}$, for all $n\in\IN$, is a computable nondecreasing sequence
of rational numbers converging to~$2^{-A}$.
By Proposition~\ref{prop:index-function} there exists a 
computable, increasing function $r\colon\IN\to\IN$ such that, for 
infinitely many $n$, $2^{-A} - a_{r(n)} < 2^{-n}$.
We obtain
$\{0,\ldots,n-1\} \cap A \subseteq \mathrm{Enum}(f)[r(n)]$ for infinitely many $n$.
Hence, $f$~is $r$-good.

\smallskip

$(4) \Rightarrow (3)$:
Trivial, as every c.e.~set has a computable enumeration.

\smallskip

$(2) \Rightarrow (3)$:
Trivial.

\smallskip

$(3) \Rightarrow (2)$:
Assume that ${r\colon\IN\to\IN}$ is a computable, nondecreasing, unbounded function
and  that $f\colon\IN\to\IN$ is a computable $r$-good enumeration of $A$.
If $A$ is finite, then $f$ itself is trivially an $\id_\IN$-good enumeration of $A$; so assume that $A$ is infinite.
Let $L_0, M_0:=\emptyset$ and for $n \geq 1$ define
$L_n := \{0,\ldots,n-1\}\cap \mathrm{Enum}(f)[r(n)]$
as well as $M_n := L_n \setminus L_{n-1}$; that is, $M_1,\ldots,M_n$ forms a disjoint partition of $L_n$ for every $n$.

We let $g\colon\IN\to\IN$ be the enumeration that first enumerates all elements of $M_1$ in increasing order, then those of $M_2$ in increasing order, and so on. More formally speaking, $g\colon\IN\to\IN$ is defined as follows:
For every~$n\in\IN$, let $m_n$ be the cardinality of $M_n$,
let $k_0^{(n)},\ldots,k_{m_n-1}^{(n)}$ be its elements in increasing order and, for $t$
with $0\leq t \leq m_n-1$, define
\[ g\left( t + \sum_{j<n} m_j\right) := 1+k_t^{(n)} . \]
This is well-defined due to $A$'s infinity.
Clearly, $g$ is a computable enumeration (in fact, without repetitions) of $A$ with $g(n)\neq 0$ for all $n$.
We claim that it is in fact an $\id_\IN$-good enumeration of $A$.
To see this, fix any of the, by assumption, infinitely many $n$ with
\[ \{0,\ldots,n-1\} \cap A \subseteq \mathrm{Enum}(f)[r(n)]. \]
By construction, $g$ enumerates exactly the elements of 
\[L_{n-1} = M_0\cup \dots \cup M_{n-1}\]
during the first $\sum_{j<n} m_j$ stages and, by choice of $n$, all elements of  $\{0,\ldots,n-1\}\cap A$ that are not enumerated during these stages
must be in $M_n$. In fact, they are exactly all elements of $M_n$, and thus will be enumerated by $g$ in the immediately following stages, starting with stage $\sum_{j<n} m_j$. As there cannot be more than ${n-\sum_{j<n} m_j}$ such numbers,
this process will be completed before stage~$n$.
\end{proof}

\begin{remark}
Note that ``$(3) \Rightarrow (2)$'' is the only implication in Theorem~\ref{theorem:regaining-sets} whose proof is not fully uniformly effective; in its proof we non-uniformly distinguished whether $A$ is finite or not. A simple topological argument shows that this  is unavoidable; in fact, there does not even exist a continuous function~$F\colon\IN^\IN\to\IN^\IN$ mapping every $2\id_\IN$-good enumeration without repetitions
of an arbitrary $A \subseteq\IN$ to an $\id_\IN$-good enumeration of $A$. To see this, we use the following two facts:
\begin{itemize}
	\item[(i)] There exists only one enumeration of the empty set, namely the constant zero function $\mathbf{0}\colon\IN\to\IN$, which is an $\id_\IN$-good enumeration.
	\item[(ii)] For every function $f \colon \IN \to \IN$ that is an $\id_\IN$-good enumeration of $\IN$ we have $f(n) \neq 0$ for all $n \in \IN$.
\end{itemize}
Now assume, for the sake of a contradiction, that there is an~$F$ as described. By~(i), we must have $F(\mathbf{0})=\mathbf{0}$, and as $F$ was assumed to be continuous, there exists an~$n_0\in\IN$ with
$F(0^{n_0}\IN^\IN)\subseteq 0\IN^\IN$; that is, $F$ must map all inputs starting with sufficiently many $0$'s to a sequence starting with at least one~$0$. Define $g\colon\IN\to\IN$ as an enumeration of $\IN$ ``delayed by $n_0$ stages'', that is, let
\[ g(n):=\begin{cases}
     0 & \text{if } n<n_0, \\
     1+n-n_0 & \text{if } n\geq n_0,
     \end{cases}
\]
for $n\in\IN$. Then while $g$ is a $2\id_\IN$-good enumeration without repetitions of $\IN$, we have $F(g)\in 0\IN^\IN$, and no such enumeration can be an $\id_\IN$-good enumeration of $\IN$, contradiction.
\end{remark}

\section{Computability-Theoretic Properties}
\label{section:sets-computability}

We begin this section by observing an easy splitting theorem that we will need later. Next, we will show that there are c.e.~sets that are not regainingly approximable. Finally, we will show that every c.e.~Turing degree contains a regainingly approximable set.

\begin{theorem}
\label{theorem:splitting}
For any c.e.~set $C\subseteq\IN$ there exist two disjoint, regainingly approximable
sets $A,B\subseteq\IN$ with $C=A\cup B$.
\end{theorem}
\begin{proof}
By Sacks' Splitting Theorem~\cite{Sac1963} (see, for instance, Soare~\cite[Theorem VII.3.2]{Soa1987})
there exist two disjoint, c.e., low subsets $A,B\subseteq\IN$ with $C=A\cup B$.
Low sets are not high. Hence, by Example~\ref{example:non-high_numbers}
and Theorem~\ref{theorem:regaining-sets},
$A$ and $B$ are regainingly approximable.
\end{proof}
The theorem leaves open whether the splitting 
can be done effectively. The answer is yes: in Section~\ref{section:SplittingRegularReals} we will present a uniformly effective algorithm that even works for a larger class of objects than just the c.e.\ sets.

\smallskip

As corollaries to Theorem~\ref{theorem:splitting} we obtain two easy ways to see that the regainingly approximable sets are a strict superclass of the decidable ones.
\begin{cor}
\label{cor:regaining-not-decidable}
There exists a regainingly approximable set $A\subseteq\IN$ that is not decidable.
\end{cor}
\begin{proof}[First proof]
By Example~\ref{example:non-high_numbers}, any c.e.~set that is neither decidable nor high is regainingly approximable.
\end{proof}
\begin{proof}[Second proof]
Let $C\subseteq\IN$ be an undecidable c.e.~set.
By Theorem~\ref{theorem:splitting} there are two disjoint regainingly approximable sets $A,B$ such that ${C=A\cup B}$. At least one of $A$ or $B$ must be undecidable.
\end{proof}

Next, we separate regaining approximability from computable enumerability.
\begin{theorem}
\label{theorem:ce-not-regaining}
There exists a c.e.~set $A\subseteq\IN$ that is not regainingly approximable.
\end{theorem}

\begin{proof}
Let $\varphi_0,\varphi_1,\varphi_2,\ldots$ be a standard enumeration of all possibly partial computable functions with domain and range in $\IN$.
As usual, we write $\varphi_e(n)[t]{\downarrow}$ to express that the $e$-th Turing machine
(which computes $\varphi_e$)
stops after at most~$t$~steps on input $n$.

We shall construct a computable enumeration $g\colon\IN\to\IN$ of a set $A \subseteq\IN$
such that the following
requirements $(\mathcal{R}_{e})$ will be satisfied for all $e\in\IN$:
\[\begin{array}{r@{\;}l}
(\mathcal{R}_{e})\colon & \text{if } \varphi_e \text{ is total and increasing then } \\
  & (\exists n_e\in\IN) (\forall n> n_e)
   (\{0,\ldots,n-1\} \cap A \not\subseteq \mathrm{Enum}(g)[\varphi_e(n)]) .
\end{array}\]
According to Theorem~\ref{theorem:regaining-sets} this is sufficient.

\smallskip

We construct $g$ in stages; in stage $t$ we proceed as follows: 
Define $e:=\pi_1(\pi_1(t))$ and $k:=\pi_2(\pi_1(t))$, hence,
$\langle e,k \rangle = \pi_1(t)$. Check whether the following conditions are satisfied:
\begin{eqnarray*}
&& (\forall n\leq \langle e,k+1\rangle) \ \varphi_e(n)[t]{\downarrow} \\
&\text{and}& (\forall n < \langle e,k+1\rangle) \ \varphi_e(n) < \varphi_e(n+1) \\
&\text{and}& t \geq \varphi_e(\langle e,k+1\rangle ) \\
&\text{and}& \langle e,k \rangle \not\in \mathrm{Enum}(g)[t] .
\end{eqnarray*}
If they are, set $g(t):=1+ \langle e,k \rangle$, otherwise
$g(t):=0$.

\smallskip

This completes the construction; we proceed with the verification.
It is clear that $g$ is computable and an enumeration without repetitions of some c.e.\ set $A \subseteq \IN$.
We wish to show that $\mathcal{R}_e$ is satisfied for all $e\in\IN$.
Consider an~$e\in\IN$ such that $\varphi_e$ is total and increasing, as well as a number~${n> \langle e,0 \rangle}$.
There exists a unique $k\in\IN$ with ${\langle e,k \rangle < n \leq \langle e,k+1\rangle}$.
The function~$g$ enumerates $\langle e,k\rangle$ into $A$ at some uniquely determined stage $t$,
that is, there exists exactly one $t \in \IN$ with ${g(t)=1+\langle e,k\rangle}$.
Then 
\[\langle e,k \rangle \in \mathrm{Enum}(g)[t+1] \setminus \mathrm{Enum}(g)[t].\]
Since $n \leq \langle e,k+1\rangle$, we have
$\varphi_e(n) \leq \varphi_e( \langle e,k+1 \rangle ) \leq t$, and therefore
\[  \langle e,k \rangle \not\in\mathrm{Enum}(g)[t] \supseteq \mathrm{Enum}(g)[\varphi_e(n)] .
\]
Thus $\langle e,k \rangle \in \{0,\ldots,n-1\}\cap A$ witnesses that  $\mathcal{R}_e$ is satisfied with $n_e=\langle e,0\rangle$.
\end{proof}

Finally, we show that the regainingly approximable Turing degrees are exactly the c.e.\ Turing degrees.
\begin{theorem}
For every computably enumerable set $A \subseteq \IN$ there exists a regainingly approximable set $B \subseteq \IN$ with $A \equiv_T B$.
\end{theorem}
\begin{proof}
W.l.o.g.\ we may assume that $A$ is infinite. Thus fix some computable injective function $f \colon \IN \to \IN$ with $f(\IN) = A$. 

\medskip

We will build a c.e.\ $B$ by defining a computable injective function ${g\colon\IN\to\IN}$ and letting $B:=g(\IN)$. In parallel, we will define an increasing sequence $(S_i)_i$ such that
\begin{itemize}
\item[(i)]
for all $i$ we have $\{0,\ldots,S_i-1\} \cap B\subseteq \{g(0),\ldots,g(S_i-1)\}$, which 
implies that $B$ is regainingly approximable, and such that
\item[(ii)]
$A \equiv_T (S_i)_i \equiv_T B$.
\end{itemize}
To ensure these properties, we will define $(S_i)_i$ in such a way
that, for all~$i$ and for all $t\geq S_i$, we have
on the one hand that $f(t)\geq i$ and on the other hand that $g(t)\geq S_i$. This last property will imply~(i); and concerning~(ii), informally speaking, our choice of $(S_i)_i$ means that $f$~and $g$~can enumerate ``small'' numbers only for arguments smaller than $S_i$. This will allow us to show~$A \leq_T (S_i)_i$ and~$B \leq_T (S_i)_i$. Finally, the statements $(S_i)_i \leq_T A$ and $(S_i)_i \leq_T B$ will follow from the way we define $(S_i)_i$ alongside $g$. 

\medskip

After these informal remarks, we proceed with the full proof.
The following algorithm works recursively in infinitely many stages to compute two functions $g \colon \IN \to \IN$ and ${s\colon\IN^2\to\IN}$; we write $s_i[t]$ for~$s(i,t)$. Since it will turn out below that $(s_i[t])_t$ is eventually constant for every~$i$, allowing us to define $S_i:=\lim_{t\to\infty} s_i[t]$, it is suggestive to think of $s_i[t]$ as our preliminary guess for $S_i$ at stage $t$.

\smallskip

At stage $0$ we let
\begin{equation*}
	s_i[0] := i
\end{equation*}
for all $i\in\IN$. And for every $t\in\IN$, at stage $t+1$ we define
\begin{equation*}
	g(t) := s_{f(t)}[t]
\end{equation*}
and set
\begin{equation*}
	s_i[t+1] := \begin{cases}
		s_i[t] &\text{if $i \leq f(t)$}, \\
		s_i[t] + \max\{t, g(0), \dots, g(t)\}+1 &\text{if $i > f(t)$}, \\
	\end{cases}
\end{equation*}
for all $i \in \IN$. Finally, we define $B := g(\IN)$. 

\smallskip

This ends the construction, and we proceed with the verification. First note that $B$ is clearly computably enumerable.
	
\smallskip
	
\emph{Claim 1.} For every $t\in\IN$, the sequence $(s_i[t])_i$ is increasing.
	
\smallskip
	
\emph{Proof.}
This is clear for $t=0$, and can easily be seen to hold for all other $t$ by induction.\hfill$\Diamond$
%
	
\smallskip
	
\emph{Claim 2.} For every $i \in \IN$, the sequence $(s_i[t])_{t}$ is nondecreasing and eventually constant.
	
\smallskip
	
\emph{Proof.}
By construction $(s_i[t])_t$ is nondecreasing for every $i \in \IN$.
Fix some $i\in \IN$; to see that $(s_i[t])_t$ is eventually constant, choose $t_i$  such that $f(t) \geq i$ for all $t\geq t_i$. Note that such a $t_i$ exists because $(f(t))_t$ tends to infinity.
Then, for every $t\geq t_i$, we have $s_{i}[t+1]=s_{i}[t]$, and thus
$s_{i}[t]=s_{i}[t_i]$ for all $t\geq t_i$.\hfill$\Diamond$
%
	
\smallskip
	
Define the sequence $(S_i)_i$ by $S_i:=\lim_{t\to\infty} s_i[t]$.
By Claim~1, $(S_i)_i$ is increasing.
	
\smallskip
	
\emph{Claim 3.} For every $i\in\IN$ and every $t \geq S_i$ we have $s_i[t]=S_i$.
	
\smallskip
	
\emph{Proof.} 
Assume otherwise and choose $t\geq S_i$ with $s_i[t+1]\neq s_i[t]$. Then
\begin{equation*}
	S_i\geq s_i[t+1]=s_i[t]+\max\{t, g(0), \dots, g(t)\}+1 > t \geq S_i,
\end{equation*}
a contradiction.\hfill$\Diamond$
	
\smallskip
	
\emph{Claim 4.}
For every $i \in \IN$ and every $t \geq S_i$ we have $f(t) \geq i$.
	
\smallskip
	
\emph{Proof.} 
Consider any $i \in \IN$ and some stage $t \geq S_i$. If we had $f(t) < i$, then we would obtain
\begin{equation*}
	s_i[t+1] = s_i[t]+\max\{t, g(0), \dots, g(t)\}+1 > s_i[t]=S_i,
\end{equation*}
contradicting Claim~3.\hfill$\Diamond$
	
\smallskip
	
\emph{Claim 5.}
For every $i \in \IN$ and every $t \geq S_i$ we have $g(t) \geq S_i$. 
		
\smallskip

\emph{Proof.}
Consider any $i \in \IN$ and some stage $t \geq S_i$.
By Claim~4, $f(t) \geq i$, and by Claim~1, $(s_j[t])_j$ is increasing.
Together with Claim~3, we obtain
$g(t) = s_{f(t)}[t] \geq s_i[t] = S_i$.
\hfill$\Diamond$

\smallskip

\emph{Claim 6.}
$B$ is regainingly approximable.

\smallskip

\emph{Proof.}
According to Claim 5 the function
$n \mapsto g(n)+1$ is an $\id_\IN$-good, computable enumeration of $B$, and
hence $B$ is regainingly approximable.\hfill$\Diamond$

\smallskip

\emph{Claim 7.}
The function $g$ is injective. 
	
\smallskip

\emph{Proof.}
Consider two stages $t_1 < t_2$; we need to show $g(t_1) \neq g(t_2)$.
Since $f$ is injective, we have $f(t_1) \neq f(t_2)$.
\begin{itemize}
\item
If $f(t_1) < f(t_2)$, then due to Claims~1 and~2 we have
		\begin{equation*}
			g(t_1) = s_{f(t_1)}[t_1] <  s_{f(t_2)}[t_1] \leq s_{f(t_2)}[t_2] = g(t_2).
		\end{equation*}
\item
If $f(t_1) > f(t_2)$, then $s_{f(t_1)}[t_1] > s_{f(t_2)}[t_1]$ holds due to Claim~1. If $s_{f(t_2)}[t_2] = s_{f(t_2)}[t_1]$ then
		\begin{equation*}
			g(t_2)= s_{f(t_2)}[t_2] = s_{f(t_2)}[t_1] < s_{f(t_1)}[t_1] = g(t_1). 
		\end{equation*}
		Otherwise $s_{f(t_2)}[t_2] > s_{f(t_2)}[t_1]$ due to Claim~2 and we obtain
		\begin{equation*}
			g(t_2) = s_{f(t_2)}[t_2] \geq s_{f(t_2)}[t_1] + g(t_1) + 1 > g(t_1).
		\end{equation*}
\end{itemize} 
Thus, $g$ is injective.\hfill$\Diamond$

\smallskip
	
\emph{Claim 8a.}
$A$ is computable in $(S_i)_i$.
		
\smallskip
		
\emph{Proof:}
Given oracle $(S_i)_i$,
in order to decide whether $n\in A$ for arbitrary~${n \in \IN}$, it suffices to check whether~$n$ is enumerated by $f$ before stage~${S_{n+1}}$. If not, then since by Claim~4 we have $f(t) \geq n+1$ for all~${t \geq S_{n+1}}$, we must have $n \notin A$.\hfill$\Diamond$

\smallskip

\emph{Claim 8b.}
$B$ is computable in $(S_i)_i$.

\smallskip

\emph{Proof:}
Given oracle $(S_i)_i$,
in order to decide whether $n\in B$ for arbitrary $n \in \IN$ we proceed as follows: Compute the smallest $i \in \IN$ with $S_i > n$ and check whether~$n$ is enumerated by~$g$ before stage $S_{i}$. If not, then since by Claim~5 we have $g(t) \geq S_i > n$ for all $t \geq S_{i}$, we must have $n \notin B$.\hfill$\Diamond$
	
\smallskip

\emph{Claim 9a.}
The sequence $(S_i)_i$ is computable in $A$.

\smallskip

\emph{Proof.}
To determine $S_i$ for any $i \in \IN$
it suffices to use oracle $A$ to compute the smallest $t$ such that
\[ A \cap \{0,\dots,i-1\} \subseteq \{f(0),\ldots,f(t-1)\} \]
and to output $S_i = s_i[t]$.	\hfill$\Diamond$
	
\smallskip
	
\emph{Claim 9b.}
The sequence $(S_i)_i$ is computable in $B$.
	
\smallskip
	
\emph{Proof.}
We claim that there is a recursive algorithm using oracle $B$ that computes~$(S_i)_i$.
By construction, we have $S_0 = 0$.
So suppose that for~${i \in \IN}$ the numbers $S_0, \dots, S_i$ are known.
We claim that 
\begin{enumerate}
	\item[(i)] if $S_i\in B$ and if the uniquely determined (cf.~Claim~7) number $t$ with $g(t)=S_i$ satisfies
	$f(t) = i$ and $t \geq S_i$, then $S_{i+1}=s_{i+1}[t+1]$; and that
	\item[(ii)] $S_{i+1}=s_{i+1}[S_i]$ holds otherwise.
\end{enumerate}
Assuming this claim, oracle $B$ clearly computes $S_{i+1}$.

To see (i), assume that $S_i\in B$ and
that the uniquely determined $t$ with $g(t)=S_i$ satisfies
$f(t) = i$ and $t \geq S_i$.
By Claim 4 and because $f$~is injective, for all $t'\geq t+1$ we obtain $f(t')\geq i+1$. This implies $S_{i+1}=s_{i+1}[t+1]$.

For (ii), if $S_{i+1}\neq s_{i+1}[S_i]$, then there must exist a number $t\geq S_i$ with ${s_{i+1}[t+1]\neq s_{i+1}[t]}$ and hence with
$i+1>f(t)$. Then, by Claim 4, $f(t)=i$, and by Claim 3, $S_i=s_i[t]=s_{f(t)}(t)=g(t)$, hence, $S_i\in B$.\hfill$\Diamond$

\smallskip
	
This concludes the proof that for every c.e.\ set $A \subseteq \IN$ there exists 
a regainingly approximable set $B \subseteq \IN$ with $A \equiv_T B$.
\end{proof}

\begin{cor}
There is a high regainingly approximable set.
\end{cor}

\section{Complexity of Regainingly Approximable Numbers}
\label{section:Kolmogorov-complexity}

Let $\Sigma:=\{0,1\}$, let $\Sigma^n$ denote the set
of finite binary strings of length~$n$, and write $\Sigma^*$ for $\bigcup_n \Sigma^n$. 
For a binary string $v$, let $|v|$ denote its length.
By $0^n$ we denote the string of length $n$ that consists of $n$ zeros,
and for any infinite binary sequence $A$ and any number $n\in\IN$,
by $A \restriction n$ we denote the string $A(0)\ldots A(n-1)$ of length $n$.
For a function $f\colon\dom(f)\to\Sigma^*$ with $\dom(f) \subseteq\Sigma^*$ write \[C_f(w):=\min(\{|v|\colon f(v)=w\}\cup \{\infty\})\] for every~${w\in \Sigma^*}$. If $\dom(f)$ has the property that no two different elements in it can be prefixes of each other then we say that $f$ has prefix-free domain; and in this case it is customary to write $K_f$ for~$C_f$. 
Let $C$ denote~$C_{U_1}$ for some optimally universal Turing machine~$U_1$ and let $K$ denote~$K_{U_2}$ for some optimally universal prefix-free Turing machine~$U_2$ (see, for instance, Downey and Hirschfeldt~\cite[Sections 3.1 and 3.5]{DH2010} for a discussion of such Turing machines).
In the remainder of this section we will write $U$ for~$U_2$.
The functions $C$ and $K$ are called the {\em plain} and the {\em prefix-free Kolmogorov complexity}, respectively.
 
An infinite binary sequence $A$ is called \emph{i.o.~$K$-trivial} if there exists a constant
$c\in\IN$ such that, for infinitely many~$n$,
\[K(A \restriction n) \leq K(0^n) + c.\]
Recalling the conventions laid out in Section~\ref{dfnasdnsdfdfg}, we can make the following observation.
\begin{prop}
	\label{prop:inf-often-K-trivial}
	Every regainingly approximable number $\alpha$ is i.o.\  $K$-trivial, and
	hence not Martin-L\"of random.
\end{prop}

\begin{proof}
	Assume w.l.o.g.\ that $\alpha$ is not computable and $\alpha \in (0,1)$.
%
	Let $(a_n)_n$ be an increasing, computable sequence of w.l.o.g.\ strictly positive rational numbers
	converging to $\alpha$ such that, for infinitely many $n$, $\alpha-a_n<2^{-n}$.
	
	For every $n\in\IN$, let $u_n$ be the binary string of length $n$
	with 
	\[0.u_n \leq a_n < 0.u_n+2^{-n}.\]
	If $u_n\neq 1^n$ let $v_n\in\Sigma^n$ be such that $0.v_n = 0.u_n+2^{-n}$; otherwise let~${v_n:=u_n}$. Clearly, $(u_n)_n$ and $(v_n)_n$ are computable.
	
	Define a computable function $f\colon \dom(f)\to\Sigma^*$ with prefix-free domain $\dom(f)\subseteq\Sigma^*$ as follows:
	If $w\in\Sigma^*$ satisfies $U(w)=0^n$, for some $n\in\IN$, then  let
	$f(w0) := u_n$ and $f(w1) := v_n$. Otherwise we leave $f(wa)$ undefined for~${a\in\Sigma}$. Note that the domain of $f$ is prefix-free.
	
	By construction, for infinitely many $n$, we have $\alpha - 0.u_n < 2 \cdot 2^{-n}$. Then, for these $n$, we have $\alpha \restriction n \in \{u_n,v_n\}$; thus 
	\[K(\alpha\restriction n) \leq^+ K_f(\alpha\restriction n) = K(0^n)+1,\] and $\alpha$ is i.o.\  $K$-trivial.
\end{proof}
In view of Proposition~\ref{prop:inf-often-K-trivial}, it is natural to wonder whether regainingly approximable numbers can at the same time have infinitely many initial segments of high Kolmogorov complexity. The answer is yes, as demonstrated by the next theorem.

\begin{theorem}
\label{theorem:high-Kolmogorov-complexity}
There is a regainingly approximable number ${\alpha \in (0,1)}$ such that
$K(\alpha \restriction n) > n$ for infinitely many $n$.
\end{theorem}
We point out that an analogous result cannot hold for plain Kolmogorov complexity; this is due to a result of 
Martin-L\"of~\cite{MR451322} (see Nies, Stephan and Terwijn~\cite[Proposition~2.4]{MR2140044})
who proved that any infinite binary sequence $A$ for which ${C(A\restriction n)>n}$ holds for infinitely many~$n$ is Martin-L\"of random;
thus such an $A$ cannot be regainingly approximable.

\begin{proof}[Proof of Theorem~\ref{theorem:high-Kolmogorov-complexity}]
Fix a computable injective function $h\colon \IN\to\Sigma^*$ with ${h(\IN)=\dom(U)}$.
For $t\in\IN$ and $v\in\Sigma^*$ write
\[ U^{-1}\{v\}[t] := \{u \in\Sigma^* \colon U(u)=v \text{ and } (\exists s < t) (h(s)=u) \} \]
\[ K(v)[t] := \begin{cases}
     \infty & \text{if } U^{-1}\{v\}[t]=\emptyset, \\
     \min\{|u|\colon u \in U^{-1}\{v\}[t]\} & \text{otherwise}.
     \end{cases}
\]
Then, for every $v\in\Sigma^*$, the function $t\mapsto K(v)[t]$ is nonincreasing and eventually constant
with limit $K(v)$.

\smallskip

We sketch the underlying idea before giving the formal proof. In order to ensure that $\alpha$ has initial segments of high prefix-free Kolmogorov complexity, 
we want to mimic Chaitin's $\Omega$, that is, the sum $\sum_n 2^{-|h(n)|}$.
Of course, overdoing this would lead to a Martin-L\"of random number in the limit, and such numbers cannot be regainingly approximable. Instead, $\alpha$ is obtained from the series $\sum_n 2^{-|h(n)|}$ by multiplying its elements by
certain weights.
More precisely, at any given time some weight is fixed, and we keep adding elements to the series scaled by this weight. As continuing to do this forever would lead to a scaled copy of $\Omega$, that is, to some Martin-L\"of random number, that process must eventually produce a rational number with an initial segment of high Kolmogorov complexity.

If at a stage $t$ it looks as if that has happened, we change the weight to a new, smaller value that is at most $2^{-t}$. As all new terms that are added to the series after $t$ will now be scaled by this new weight, their total sum cannot exceed $2^{-t}$. Doing this for infinitely many $t$ then ensures regaining approximability of~$\alpha$.

One issue with this approach is that we can never be sure about the stages when we change weights. This is because prefix-free Kolmogorov complexity is uncomputable, and we have to work with its approximations. Thus, it may turn out that what looked like an initial segment of high complexity at stage~$t$ really has much lower complexity, and as a result we may have dropped the weight too early. However, since $\alpha$ only needs to be left-computable, it is easy to deal with this: we simply retroactively scale up the weights of all terms that had been given too small a weight.

\smallskip

We now come to the formal construction, which will proceed in stages to define computable functions ${\ell,r\colon\IN^2\to\IN}$ and $w\colon \IN^2\to\IN\cup\{\infty\}$
as well as a computable sequence $(a_t)_t$ of dyadic rational numbers in the interval $[0,1)$. We shall write $\ell(n)[t]$ for $\ell(n,t)$, $r(n)[t]$ for $r(n,t)$, and $w(n)[t]$ for $w(n,t)$.

\smallskip

At stage $0$ define 
\[ \ell(n)[0]:=0 \text{ and } r(n)[0]:=n \text{ and } w(n)[0]:=\infty, \]
for all~$n\in\IN$, as well as $a_0:=0$.

\smallskip

At a stage $t$ with $t>0$ first define $\ell(0)[t]:=0$ and, for $n>0$,
\[ \ell(n)[t] := \min\{m \in \IN \colon m > \ell(n-1)[t] \text{ and } K(a_{t-1} \restriction m)[t]>m\} . \]
Note that this is well-defined because for a fixed $t$ we have $K(v)[t]=\infty$ for almost all~$v\in\Sigma^*$. Intuitively speaking, $\ell(n)[t]$ is our guess at stage~$t$ about the length of what we hope will be the $n$-th witness for the existence of infinitely many initial segments of~$\alpha$ that have high complexity.

If there is no $i<t$ with $\ell(i)[t]\neq \ell(i)[t-1]$, then define
\[ r(n)[t] := r(n)[t-1] \text{ and } w(n)[t]:=w(n)[t-1], \]
for every $n\in\IN$. Otherwise let $i_t$ denote the minimal such $i$ 
and define
\[ r(n)[t] := \begin{cases}
     r(n)[t-1] & \text{if } n \leq i_t, \\
     r(n)[t-1]+t & \text{if } i_t < n,
   \end{cases}
\]
\[ w(n)[t] := \begin{cases}
     \min(w(n)[t-1],r(i_t)[t])  & \text{if } n < t,\\
     \infty & \text{otherwise,}
   \end{cases}
\]
for every $n\in\IN$. 
In either case, using the convention $2^{-\infty}:=0$, define 
\[ a_t := \sum_{n=0}^{t-1} 2^{-w(n)[t]} \cdot 2^{-|h(n)|} . \]
Using the terminology from the informal proof sketch above, $r$ is employed to reduce weights when we believe that we discovered an initial segment of high complexity, with the intent of ensuring regaining approximability of~$\alpha$. This
is then used to define the function~$w$ determining the scaling factors that will be applied to the elements of the series $\sum_n 2^{-|h(n)|}$; note how the appearance of~$i_t$ in the definition of~$w$ enables retroactively increasing these factors later if required.

\smallskip

This ends the description of the construction; we proceed with the verification. 
It is easy to see that $\ell,r, w$, and $(a_t)_t$ are computable; we need to argue that $\alpha:=\lim_{t\to\infty} a_t$ exists with the desired properties. The following claims are immediate.

\smallskip

\emph{Claim 1.} 
For every $t\in\IN$ the function $n\mapsto r(n)[t]$ is increasing.\hfill$\Diamond$

\smallskip

\emph{Claim 2.} 
For every $n\in\IN$ the function $t\mapsto r(n)[t]$ is nondecreasing.\hfill$\Diamond$

\smallskip

\emph{Claim 3.} 
For every $t\in\IN$ the function $n\mapsto w(n)[t]$ is nondecreasing
and satisfies $w(n)[t]=\infty$ for all $n\geq t$.\hfill$\Diamond$

\smallskip

\emph{Claim 4.} 
For every $n\in\IN$ the function $t\mapsto w(n)[t]$ is nonincreasing.\hfill$\Diamond$

\smallskip

Due to Claim 3, using the convention $2^{-\infty}=0$, we can write
\[ a_t = \sum_{n=0}^\infty 2^{-w(n)[t]} \cdot 2^{-|h(n)|} , \]
for all $t\in\IN$. 
By Claim 4, $(a_t)_t$ is nondecreasing, and since 
\[ a_t \leq \sum_{n=0}^{t-1} 2^{-|h(n)|} < 1 , \]
for all $t\in\IN$, its limit $\alpha$ exists and 
is a left-computable number in $(0,1]$.

\smallskip

\emph{Claim 5.}
Let $n_0,t_0\in\IN$ be numbers such that
\begin{itemize}
	\item[(i)] $n_0<t_0$,
	\item[(ii)] $\ell(n_0)[t_0]\neq \ell(n_0)[t_0-1]$,
	\item[(iii)] for every $t\geq t_0$ and every $i<n_0$, $\ell(i)[t]=\ell(i)[t-1]$.
\end{itemize}
Then $w(n)[t]=w(n)[t_0]$ for all $n<t_0$ and all $t\geq t_0$.

\smallskip

\emph{Proof.} Fix any $n<t_0$ and $t>t_0$; in order to show ${w(n)[t]=w(n)[t_0]}$ we may inductively assume that the statement has already been proven for $t-1$, that is, that ${w(n)[t-1]=w(n)[t_0]}$ holds; thus it suffices to prove $w(n)[t]=w(n)[t-1]$.

This is clear by construction if there is no $i<t$ with ${\ell(i)[t]\neq \ell(i)[t-1]}$; thus let us assume that such an $i$ does exist. In this case, by construction, $w(n)[t]=\min(w(n)[t-1],r(i_t)[t])$ and hence it is enough to show that $w(n)[t-1] \leq r(i_t)[t]$.

By definition we have $\ell(i_t)[t]\neq\ell(i_t)[t-1]$; thus (iii) implies~$i_t\geq n_0$. Similarly, (ii) and (iii) imply $i_{t_0}=n_0$. Then using Claims 1 and 2 as well as the induction hypothesis we obtain
\begin{align*}
 w(n)[t-1] &= w(n)[t_0] = \min(w(n)[t_0-1],r(i_{t_0}[t_0])) \\
 &\leq  r(i_{t_0})[t_0] = r(n_0)[t_0] \leq r(i_t)[t_0] \leq r(i_t)[t], 
\end{align*}
which concludes the proof.\hfill$\Diamond$

\smallskip

\emph{Claim 6.}
For every $n\in\IN$, the function $t\mapsto \ell(n)[t]$ is eventually constant.

\smallskip

\emph{Proof.}
For the sake of a contradiction, assume that there is a
smallest $n_0\in\IN$ such that
$t\mapsto \ell(n_0)[t]$ is not eventually constant.
By the definitions above we must have $n_0>0$. Define
\[ t_0 := \min\left\{t > n_0 \colon \begin{array}{c}
      \ell(n_0)[t]\neq \ell(n_0)[t-1] \text{ and } \\
      (\forall i< n_0) (\forall s \geq t) ( \ell(i)[s]=\ell(i)[s-1] )
      \end{array} \right\}
\]
and, for $j>0$,
\[ t_j := \min\{t > t_{j-1} \colon \ell(n_0)[t]\neq \ell(n_0)[t-1] \} .
\]
The sequence $(t_j)_j$ is trivially increasing and
$w(n)[t_j]=w(n)[t_0]$ for all~$n<t_0$ and all $j\in\IN$, by Claim 5.

We show by induction that 
$w(n)[t]\geq r(n_0)[t_0]$ for all $n\geq t_0$ and $t\geq t_0$: 
\begin{itemize}
	\item For $t=t_0$ this is clear by Claim 3.
	
	\item For $t>t_0$ and if there is no $i<t$ with $\ell(i)[t]\neq\ell(i)[t-1]$ then the assertion is clear inductively by definition of~$w$.
	
	\item  If $t>t_0$ and there is an $i<t$ with $\ell(i)[t]\neq\ell(i)[t-1]$ then,
	\begin{itemize}
		\item if $n\geq t$ we have $w(n)[t]=\infty > r(n_0)[t_0]$, and
		\item if $n<t$ then
		$w(n)[t] = \min(w(n)[t-1],r(i_t)[t]) \geq r(n_0)[t_0]$.

		This is because, on the one hand, $w(n)[t-1] \geq r(n_0)[t_0]$ can be assumed to hold inductively and, on the other hand, 
		$r(i_t)[t] \geq r(n_0)[t] \geq r(n_0)[t_0]$
		by Claims 1 and 2 together with the fact that $i_t\geq n_0$ by choice of~$t_0$.
	\end{itemize}	
\end{itemize}

By choice of $t_0$ we have for every $j\in\IN$ that $i_{t_j}=n_0$ and thus \[r(i_{t_j})[t_j]=r(n_0)[t_j]=r(n_0)[t_0].\]
Combining this with the previous results and using the definition of~$w$, for $j>0$ we see that
\[ w(n)[t_j]= r(n_0)[t_0] \text{ for } t_0 \leq n < t_j. \]
Hence, for all $j>0$ we obtain
\begin{eqnarray*}
  a_{t_j} &=& \sum_{n=0}^{t_j-1} 2^{-w(n)[t_j]} \cdot 2^{-|h(n)|} \\
    &=& \sum_{n=0}^{t_0-1} 2^{-w(n)[t_0]} \cdot 2^{-|h(n)|} + \sum_{n=t_0}^{t_j-1} 2^{-r(n_0)[t_0]} \cdot 2^{-|h(n)|} \\
    &=& b_0 + 2^{-r(n_0)[t_0]} \cdot \sum_{n=0}^{t_j-1} 2^{-|h(n)|} ,
\end{eqnarray*}
where $b_0$ is the rational number given by
\[ b_0 := \sum_{n=0}^{t_0-1} \left( 2^{-w(n)[t_0]} - 2^{-r(n_0)[t_0]} \right) \cdot 2^{-|h(n)|} . \]
Writing $\Omega:=\sum_{u\in\dom(U)} 2^{-|u|}=\sum_{n=0}^\infty  2^{-|h(n)|}$
we conclude that
\[ \alpha = b_0 + 2^{-r(n_0)[t_0]} \cdot  \Omega .\]
Thus we see that $\alpha$ is Martin-L\"of random, and by a result of Chaitin~\cite{Cha1987} (see, for instance, Downey and Hirschfeldt~\cite[Corollary 6.6.2]{DH2010}) we have for all sufficiently large $m$ that $K(\alpha \restriction m) > m$.
Fix an $m_0>\ell(n_0-1)[t_0]$ with $K(\alpha \restriction m_0) > m_0$, and a~$j_0$ such that, for all $t\geq t_{j_0}$, we have $a_{t-1} \restriction m_0 = \alpha \restriction m_0$.
Then, for $t\geq t_{j_0}$ we obtain
\[ K(a_{t-1} \restriction m_0)[t]
  = K(\alpha \restriction m_0)[t]
  \geq K(\alpha \restriction m_0)
  > m_0, \]
hence, by definition, $\ell(n_0)[t]\leq m_0$ for any such~$t$. 
Combining this with the fact that the map $t \mapsto \ell(n_0)[t]$ is by definition nondecreasing for~${t\geq t_{j_0}}$ shows that it must be 
eventually constant after all, a contradiction.\hfill$\Diamond$

\smallskip

Define $L_n:=\lim_{t\to\infty} \ell(n)[t]$, for $n\in\IN$.
By the definition of $\ell$ we have that $(L_n)_n$
is increasing and that 
\begin{equation*}
\label{eq:InfinitelyOftenHighKolmogorovComplexity}
K(\alpha\restriction L_n) > L_n \text{ for all } n\in\IN.\tag{$\ast$}
\end{equation*}
It remains to show that $\alpha$ is regainingly approximable.
To that end, define a sequence $(s_m)_m$ by
\[ s_m := \max\left(\{0\}\cup
    \{s \in\IN \colon (\exists i<m)\; (i<s \text{ and } \ell(i)[s]\neq \ell(i)[s-1]) \}\right), \]
for all $m\in \IN$. Clearly, $(s_m)_m$ is nondecreasing.

\smallskip

{\em Claim 7.}
The sequence $(s_m)_m$ is unbounded.

\smallskip

{\em Proof.}
For the sake of a contradiction, let us assume that $(s_m)_m$ is bounded and that $S\in\IN$ is an upper bound.
Then, for all $t>S$ and all~${m\in\IN}$ there is no $i<m$ with $i<t$ and $\ell(i)[t]\neq \ell(i)[t-1]$.
Hence, for all $t>S$ and $i<t$ we have $\ell(i)[t]=\ell(i)[t-1]$.
This implies $w(n)[t]=w(n)[t-1]$ for all $n\in\IN$ and all $t>S$.
We obtain $a_t=a_{t-1}$ for all $t>S$, hence, $\alpha=a_S$.
Thus $\alpha$ would be a rational number, which contradicts~\eqref{eq:InfinitelyOftenHighKolmogorovComplexity}.
\hfill$\Diamond$

\smallskip

Consider an $m\in \IN$ with $s_m>0$.
We claim that 
\[ \alpha - a_{s_m} < 2^{-s_m} . \]
By Claim 4, for every $n\in\IN$, the function $t\mapsto w(n)[t]$ is eventually constant.
Thus we can define $W_n:=\lim_{t\to\infty} w(n)[t]$ for every $n\in\IN$ and write
\[ \alpha = \sum_{n=0}^\infty 2^{-W(n)} \cdot 2^{-|h(n)|} . \]

\smallskip

\emph{Claim 8.} For $n < s_m$ we have $W(n)= w(n)[s_m]$.

\smallskip

\emph{Proof.} 
By definition of $s_m$, we have $\ell(i)[t] = \ell(i)[t-1]$ for all $t\geq s_m$ and all $i<i_{s_m}$. Thus, the claim follows from Claim 5 with $s_m$ in place of $t_0$ and $i_{s_m}$ in place of $n_0$.
\hfill$\Diamond$

\smallskip

\emph{Claim 9.} For $n \geq s_m$ we have $W(n) \geq s_m$.

\smallskip

\emph{Proof.}
It suffices to prove that $w(n)[t]\geq s_m$ for all $t\geq s_m$ and $n\geq s_m$.
We proceed by induction over $t$. By Claim 3 the assertion is clear for $t=s_m$.
Thus consider an arbitrary $t>s_m$. If there is no $i<t$ with $\ell(i)[t]\neq\ell(i)[t-1]$, then the assertion is clear inductively by definition of~$w$. 

So let us assume that there exists such an $i$. In case that $n\geq t$ we obtain $w(n)[t]=\infty>s_m$, and we are done. Otherwise, if $n<t$, then by definition $w(n)[t]=\min(w(n)[t-1],r(i_t)[t])$. Assuming inductively that the assertion is already proven for $t-1$, it is thus enough to show $r(i_t)[t]\geq s_m$. By the definition of $s_m$ we have $i_t>i_{s_m}$, thus by Claim~2,
\[ r(i_t)[t] \geq r(i_t)[s_m] = r(i_t)[s_m-1]+s_m \geq s_m, \]
completing the proof of Claim 9.\hfill$\Diamond$

\smallskip

Using Claim 8 it follows that 
\[a_{s_m} = \sum_{n=0}^{s_m-1} 2^{-w(n)[s_m]} \cdot 2^{-|h(n)|}
= \sum_{n=0}^{s_m-1} 2^{-W(n)} \cdot 2^{-|h(n)|},\] 
and thus, using Claim 9, we can conclude that
\begin{align*}
\alpha - a_{s_m} &= \sum_{n=s_m}^\infty 2^{-W(n)} \cdot 2^{-|h(n)|}\\
   &\leq \sum_{n=s_m}^\infty 2^{-s_m} \cdot 2^{-|h(n)|}\\
   &< 2^{-s_m} \cdot \Omega \\
   &< 2^{-s_m}.\qedhere\end{align*} 
\end{proof}

\section{Complexity of Regainingly Approximable Sets}\label{sdfhsdfhjsdfdgdgddf}

As every c.e.\ set is i.o.\  $K$-trivial by a result of Barmpalias and Vlek~\cite[Proposition 2.2]{MR2866562}, this holds in particular for every regainingly approximable set. Thus, in analogy to the results in the last section, it is natural to wonder if a regainingly approximable set may at the same time have infinitely many initial segments of high Kolmogorov complexity. 

Of course, in the setting of c.e.\ sets, there is no hope of achieving complexities as high as in the previous section; this is because for any such set $A$ we trivially have $C(A \restriction n ) \leq^+ 2 \log n$ for all $n$. Kummer~\cite[Theorem 3.1]{Kummer1996} showed that there exist c.e.~sets that achieve this trivial upper bound infinitely often; to be precise, that there exists a c.e.~set $A$ and a constant $d\in\IN$ such that, for infinitely many $n$,
\begin{equation}
\label{eq:Kummercomplex}
C(A \restriction n ) \geq 2 \log n - d. \tag{$\dagger$}
\end{equation}
Such sets have been named {\em Kummer complex}~(see, for instance, Downey and Hirschfeldt~\cite[Section~16.1]{DH2010}). 
We will show that there are Kummer complex regainingly approximable sets.

For prefix-free complexity, analogous observations can be made: In this setting it is easy to see that for
every computably enumerable set $A$ we have
$K(A \restriction n) \leq^+ 2\log(n) + 2\log\log(n)$
for all $n \in \IN$. 
Barmpalias and Downey~\cite[Theorem 1.10]{BD2017} constructed a c.e.\ set $A$ and a number $d\in\IN$ with 
\[K(A \restriction n) \geq 2\log(n) + \log\log(n) - d\]
for infinitely many $n \in \IN$; for the purposes of this article we will call such sets \emph{Kummer $K$-complex}. Again, we will show that there are Kummer $K$-complex 
regainingly approximable sets. 

Thus, in this sense, regainingly approximable sets can achieve  both maximal plain and maximal prefix-free complexity
infinitely often.

%

\begin{theorem}
	\label{theorem:KummerComplexRegaininglyApproximableSet}
	There exists a regainingly approximable set that is Kummer complex.
\end{theorem}
We will use the following lemma in the proof. 
\begin{lem}
	\label{lemma:union-KolmogorovComplexity}
	For any c.e.~sets $X, Y \subseteq \IN$ there exists some $d \in \IN$ with
\begin{equation}
\label{eq:union-KolmogorovComplexity}
		C((X\cup Y)\restriction n) \leq \max\{C(X\restriction n), C(Y\restriction n)\} + d\tag{$\ddagger$}
\end{equation}
	for all $n \in \IN$.
\end{lem}
In the remainder of this section we write $U$ for the Turing machine $U_1$ that was used to define plain Kolmogorov complexity.
\begin{proof}[Proof of Lemma~\ref{lemma:union-KolmogorovComplexity}]
	Let $f_X\colon\IN\to\IN$ and $f_Y\colon\IN\to\IN$ be computable enumerations (in the sense defined in Section~\ref{section:sets}) of $X$ and $Y$, respectively.
	Then $f_Z\colon\IN\to\IN$ defined by
	\[ f_Z(n):=\begin{cases} 
		f_X(n/2) & \text{if } n \text{ is even}, \\
		f_Y((n-1)/2) & \text{if } n \text{ is odd},
	\end{cases}
	\]
	is a computable enumeration of $Z:=X\cup Y$.
	Define a computable function $g\colon \dom(g)\to\Sigma^*$ with $\dom(g)\subseteq\Sigma^*$ as follows
for every $v\in\Sigma^*$: 
	\begin{itemize}
		\item
		If $U(v)$ is defined and if there is a $t\in\IN$ such that 
		\[U(v) = \mathrm{Enum}(f_X)[t] \restriction |U(v)|,\]
		then let $t$ be the smallest such number and set 
		\[g(v0):= \mathrm{Enum}(f_Z)[2t] \restriction |U(v)|.\]
		Otherwise we leave $g(v0)$ undefined.
		\item
		If $U(v)$ is defined and if there is a $t\in\IN$ such that 
		\[U(v) = \mathrm{Enum}(f_Y)[t] \restriction |U(v)|,\]
		then let $t$ be the smallest such number and set 
		\[g(v1):= \mathrm{Enum}(f_Z)[2t] \restriction |U(v)|.\]
		Otherwise we leave $g(v1)$ undefined.	
	\end{itemize}
	
	For the verification, consider any $n \in \IN$.
	Let $s_X\in \IN$ be the smallest number with
	$\mathrm{Enum}(f_X)[s_X]\restriction n = X\restriction n$, and
	let $s_Y\in \IN$ be the smallest number with
	$\mathrm{Enum}(f_Y)[s_Y]\restriction n = Y\restriction n$. Then 
	$s:=\max\{s_X,s_Y\}$ satisfies
	$\mathrm{Enum}(f_Z)[2s]\restriction n = Z\restriction n$.
	
	If $s=s_X$, then for every $v\in\Sigma^*$ with $U(v)=X\restriction n$ we obtain $g(v0)=Z\restriction n$.
	Similarly, if $s=s_Y$, then for every $v\in\Sigma^*$ with $U(v)=Y\restriction n$ we obtain $g(v1)=Z\restriction n$.
	Thus, 
	\[C(Z\restriction n) \leq^+ C_g(Z\restriction n) \leq \max\{C(X\restriction n), C(Y\restriction n)\} + 1 .\qedhere\]
%
\end{proof}
Note that the proof works equally well with prefix-free complexity in place of plain complexity, leading to the following corollary.
\begin{cor}
	\label{lemma:union-KolmogorovComplexity_pf}
	For any c.e.~sets $A, B \subseteq \IN$ there exists some $d \in \IN$ with
	\[
		K((A\cup B)\restriction n) \leq \max\{K(A\restriction n), K(B\restriction n)\} + d
	\]
	for all $n \in \IN$.\qed
\end{cor}

\begin{proof}[{Proof of Theorem~\ref{theorem:KummerComplexRegaininglyApproximableSet}}]
	Let $Z \subseteq \IN$ be a c.e.\ Kummer complex set.
	By Theorem~\ref{theorem:splitting} 
	there exist regainingly approximable sets ${X, Y \subseteq Z}$ with $X \cup Y = Z$.
	Fix a constant $d \in \IN$ that is large enough to witness the Kummer complexity~\eqref{eq:Kummercomplex} of~$Z$
	and such that~\eqref{eq:union-KolmogorovComplexity} is true for all $n$.
	If we let $c:=2d$, then we have, for the infinitely many $n$ satisfying~\eqref{eq:Kummercomplex}, that
	\[ 2\log(n) - d
	\leq C(Z \restriction n) 
	\leq \max\{C(X \restriction n), C(Y \restriction n) \} + d,
	\]
	hence,
	\[ \max\{C(X \restriction n), C(Y \restriction n) \} \geq 2\log(n) - c, \]
	and at least one of $X$ and $Y$ has to be Kummer complex.
\end{proof}

By applying the same proof to a set $Z$ that is Kummer $K$-complex instead of Kummer complex and by replacing the use of Lemma~\ref{lemma:union-KolmogorovComplexity} by that of Corollary~\ref{lemma:union-KolmogorovComplexity_pf}, we can also obtain the following result.
\begin{theorem}
 	\label{theorem:KummerKComplexRegaininglyApproximableSet}
	There exists a regainingly approximable set that is Kummer $K$-complex.\qed 
\end{theorem}

\section{Arithmetical Properties}
\label{section:arithmetical-properties}

The class of regainingly approximable sets is not closed under union, according to 
Theorems~\ref{theorem:splitting} and~\ref{theorem:ce-not-regaining}. The following limited closure properties do hold, however, and will be useful in the proof
of the next theorem.

\begin{lem}\,
\label{lemma:union-dec--image}
\begin{enumerate}
\item
\label{lemma:union-dec--image-1}
The union of a regainingly approximable set and a decidable set is regainingly approximable.
\item
\label{lemma:union-dec--image-2}
If $A$ is a regainingly approximable set and $f\colon \IN\to\IN$ is a computable, nondecreasing
function, then the set
\[f(A):=\{n\in\IN \colon (\exists k\in A)\; n=f(k)\}\]
is regainingly approximable.
\end{enumerate}
\end{lem}

\begin{proof}
Let $A\subseteq\IN$ be a regainingly approximable set.
By Theorem~\ref{theorem:regaining-sets} there exists a computable
$\id_\IN$-good enumeration $g\colon \IN\to\IN$ of $A$.
For the first assertion, let $B\subseteq\IN$ be a decidable set.
Then the function $h\colon \IN\to\IN$ defined by
$h(2n):=g(n)$ and
$h(2n+1):=n+1$ if $n\in B$, $h(2n+1):=0$ if $n\not\in B$, 
is a computable and $2n$-good enumeration of~${A\cup B}$.

For the second assertion, let $f\colon \IN\to\IN$ be a computable, nondecreasing function.
Then the function $h\colon \IN\to\IN$ 
defined by
\[ h(n):=\begin{cases}
     0 & \text{if } g(n)=0, \\
     f(g(n)-1)+1 & \text{if } g(n)>0,
     \end{cases} \]
for $n\in\IN$, is a computable enumeration of $f(A)$.
If $f$ is bounded then $f(A)$ is finite
and $h$ is trivially an $\id_\IN$-good enumeration of $f(A)$.
Thus assume that $f$ is unbounded.
Then the function $r\colon \IN\to\IN$ defined by
\[ r(n) := \max\{m\in\IN \colon f(m)\leq n\} , \]
for $n\in\IN$, is computable, nondecreasing, and unbounded. 
We claim that $h$ is an $r$-good enumeration of $f(A)$.
By assumption, the set
\[ B := \{m\in\IN \colon \{0,\ldots,m-1\}\cap A \subseteq \mathrm{Enum}(g)[m]\}  \]
is infinite. So is the set $C:=f(B)$. 
Consider numbers $n\in C$ and $m\in B$ with $f(m)=n$.
Then $m \leq r(n)$ and we obtain
\begin{align*}
\{0,\ldots,n-1\}\cap f(A)
&= \{0,\ldots,f(m)-1\}\cap f(A) \\
&\subseteq  f(\{0,\ldots,m-1\}\cap A) \\
&\subseteq  f(\mathrm{Enum}(g)[m]) \\
&= \mathrm{Enum}(h)[m] \\
&\subseteq \mathrm{Enum}(h)[r(n)].\qedhere
\end{align*}
\end{proof}

The class of regainingly approximable sets is also not closed under intersection, as the following theorem establishes.
\begin{theorem}
\label{theorem:intersection}
There exist regainingly approximable sets $A,B\subseteq\IN$ such that $A\cap B$
is not regainingly approximable.
\end{theorem}
\begin{proof}
For natural numbers $a,b$ and a set $D\subseteq\IN$ we write $(a\cdot D + b)$ for the
set
\[ (a\cdot D + b) :=\{ n\in\IN \colon (\exists d\in D)\; n=a\cdot d + b  \}  \]
and $(a\cdot D)$ for the set $(a\cdot D + 0)$.
By Theorem~\ref{theorem:ce-not-regaining} there exists a c.e.~set $\widetilde{C}\subseteq\IN$
that is not regainingly approximable.
By Theorem~\ref{theorem:splitting} there exist two disjoint, regainingly approximable sets
$\widetilde{A},\widetilde{B}\subseteq\IN$ with $\widetilde{A} \cup\widetilde{B} = \widetilde{C}$.
By Lemma~\ref{lemma:union-dec--image} the sets 
\[ A:=(2\cdot \widetilde{A}) \cup (2\cdot\IN+1) \ \text{ and } \ 
   B:=(2\cdot\widetilde{B}+1) \cup (2\cdot \IN)  \]
are regainingly approximable.
We claim that their intersection 
\[A\cap B = (2\cdot \widetilde{A}) \cup (2\cdot \widetilde{B}+1)\]
is not regainingly approximable. To see this, let $g\colon \IN\to\IN$ be defined by $g(n)= \lfloor n/2 \rfloor$ for all $n\in \IN$.
We observe ${\widetilde{C}=g(A\cap B)}$.
Thus, if $A\cap B$ were a regainingly approximable set, then
so would be $\widetilde{C}$ according to Lemma~\ref{lemma:union-dec--image}~(\ref{lemma:union-dec--image-2}), a contradiction.
\end{proof}


We return to the study of regainingly approximable numbers.
\begin{cor}\,
		\label{cor:slc-regaining-not-regaining}
		\begin{enumerate}
			\item\label{cor:slc-regaining-not-regaining-1}
			There exists a strongly left-computable number that is not regainingly approximable.
			\item\label{cor:slc-regaining-not-regaining-2}
			There exists a strongly left-computable number that is regainingly approximable but not computable.
		\end{enumerate}
	\end{cor}
\begin{proof}
The first assertion follows from Theorem~\ref{theorem:ce-not-regaining}
and Theorem~\ref{theorem:regaining-sets}.
The second assertion follows from Corollary~\ref{cor:regaining-not-decidable}
and Theorem~\ref{theorem:regaining-sets}
and from the well-known fact that, for any $A\subseteq \IN$,
the number $2^{-A}$ is computable if and only if
$A$ is decidable.
\end{proof}

From now on, let $\leqS$ denote Solovay~\cite{Sol1975} reducibility (see, for instance, Downey and Hirschfeldt~\cite[Section~9.1]{DH2010}) between left-computable numbers. The following result states that the regainingly approximable numbers are closed downwards with respect to this reducibility.
\begin{prop}
\label{prop:reg-Solovay}
Let $\beta$ be a regainingly approximable number, 
and let $\alpha$ be a left-computable number with $\alpha \leqS \beta$.
Then $\alpha$ is regainingly approximable as well.
\end{prop}
\begin{proof}
Let $f\colon \{q\in \IQ \colon q < \beta\}\to\IQ$
be a computable function and ${c\in\IN}$~be a number such that,
for all $q\in\{q\in \IQ \colon q < \beta\}$, we have $f(q)<\alpha$ and ${\alpha - f(q) < 2^c \cdot (\beta - q)}$.
By Proposition~\ref{prop:equivalent-conditions} there exists a 
computable and increasing sequence $(b_n)_n$ of rational numbers converging to $\beta$
such that $\beta - b_n < 2^{-n-c}$ for infinitely many~${n\in\IN}$.
The sequence $(a_n)_n$ defined by
\[ a_n := \max\{f(b_i) \colon 0 \leq i \leq n\} \]
is a nondecreasing, computable sequence of rational numbers converging to $\alpha$.
For the infinitely many $n$ with $\beta - b_n < 2^{-n-c}$ we obtain
\[ \alpha - a_n \leq \alpha - f(b_n) < 2^c \cdot(\beta - b_n) < 2^{-n} . \qedhere \]
\end{proof}
\begin{cor}\,
\begin{enumerate}
\item
If the sum of two left-computable numbers is regainingly approximable,
then both of them are regainingly approximable.
\item
The sum of a regainingly approximable number
and a computable number is again a regainingly approximable number.
\end{enumerate}
\end{cor}
\begin{proof}
The first assertion follows from Proposition~\ref{prop:reg-Solovay} and from the fact
that for any two left-computable numbers $\alpha$, $\beta$ one has $\alpha \leqS \alpha+\beta$
and $\beta \leqS \alpha+\beta$.
Since adding a computable number to a left-computable number does not change its Solovay degree,
the second assertion follows from Proposition~\ref{prop:reg-Solovay} as well.
\end{proof}

\begin{cor}
\label{cor:sum}
Every strongly left-computable number can be written as the sum of
two strongly left-computable numbers that are regainingly approximable.
\end{cor}

\begin{proof}
This follows from Theorem~\ref{theorem:splitting} together with Theorem~\ref{theorem:regaining-sets}.
\end{proof}

\begin{cor}\label{cor:from-sum}
There exist two strongly left-computable and regainingly approximable numbers whose
sum is not regainingly approximable.
\end{cor}
\begin{proof}
According to Corollary~\ref{cor:slc-regaining-not-regaining}~(\ref{cor:slc-regaining-not-regaining-1}),
there exists a strongly left-com\-pu\-table number~$\gamma$ that is not regainingly approximable.
Then, according to Corollary~\ref{cor:sum}, there exist two strongly left-computable and
regainingly approximable numbers $\alpha$, $\beta$ with ${\alpha+\beta=\gamma}$ that witness the truth of the assertion.
\end{proof}

Corollary~\ref{cor:sum} raises the question whether every left-computable number
can be written as the sum of two regainingly approximable numbers; the answer is no.
This follows from Proposition~\ref{prop:inf-often-K-trivial},
from the fact that there exist Martin-L\"of random left-computable numbers, and from
the result of Downey, Hirschfeldt, Nies~\cite[Corollary 3.6]{DHN2002}
that the sum of two left-computable numbers that are not Martin-L\"of random is again
not Martin-L\"of random.

\section{Splitting Regular Reals}
\label{section:SplittingRegularReals}

In Theorem~\ref{theorem:splitting} we have seen that 
for any c.e.~set $C\subseteq\IN$ there exist two disjoint, regainingly approximable
sets $A,B\subseteq\IN$ with $C=A\cup B$.
That result left open whether this splitting can be obtained effectively. We now give a positive answer by presenting an according algorithm in the following lemma. We will formulate this splitting algorithm in a more general form so that we may apply it not only to the class of c.e.\ sets, but in fact to a larger class of objects, namely to the regular reals as defined by Wu~\cite{Wu2005}.

For a function $f\colon \IN\to\IN$ and a set $S\subseteq\IN$ from now on we will write 
\[ f^{-1}S:=\{i\in\IN\colon f(i)\in S\} . \]

\begin{lem}
\label{lemma:splitting}
There is an algorithm which,
given a function $f\colon \IN\to\IN$ with the property that $f^{-1}\{n+1\}$ is finite for every $n$,
computes two functions $g,h\colon \IN\to\IN$ such that for every $n$
\begin{equation}
  |g^{-1}\{n+1\}| + |h^{-1}\{n+1\}| = |f^{-1}\{n+1\}| \tag{$\ast$}
\end{equation}
and such that there exists an increasing sequence $(S_i)_i$ with
\begin{eqnarray*}
    && g^{-1}\{1,\ldots,S_i\}\subseteq \{0,\ldots,S_i-1\} \text{ for all even } i \text{ and } \\
    && h^{-1}\{1,\ldots,S_i\}\subseteq \{0,\ldots,S_i-1\}  \text{ for all odd } i .
\end{eqnarray*}
\end{lem}

\begin{proof}
Let a function $f$ as in the statement be given. We will describe an algorithm that works in stages to define the desired functions~$g$ and~$h$, as well as a function $s\colon\IN\times\IN\to\IN$. We write $s_i[t]$ for $s(i,t)$ to express the intuition that $s_i[t]$ is our current guess for~$S_i$ at stage~$t$.

\medskip

We begin with a high level overview of the proof strategy:
Whenever $f$~enumerates a number~$n$ at stage~$t$, we need to let either $g$ or $h$ enumerate it. To decide between the two, we follow a greedy strategy in the following sense: for every~$t$ the sequence $(s_i[t])_i$ will be increasing; if at stage $t$ the largest number $i$ with $s_i[t]\leq n$ is even then we let $g$ enumerate $n$, otherwise $h$. 
As we will show this is sufficient to prove the statement of the lemma; in particular, for every $i$, $(s_i[t])_t$ will turn out to be nondecreasing and eventually constant, allowing us to define ${S_i:=\lim_{t\to\infty} s_i[t]}$.

%
%

\medskip

With this we are ready to give the formal proof. 
\smallskip

At stage $0$ we define $s_i[0]:=i$ for all $i\in\IN$.

\smallskip

At a stage $t+1$ with $t\in\IN$ we proceed as follows:

\begin{itemize}
	\item If $f(t)=0$ then we set $g(t):=0$ and $h(t):=0$. Intuitively speaking this means that $f$ does not enumerate any new element at stage~$t$, and that we therefore do not let $g$ and $h$ enumerate any new elements either.
	We also set $s_i[t+1]:=s_i[t]$ for all $i\in\IN$.
	
	\item If $f(t)>0$ then this means $n:=f(t)-1$ is enumerated by $f$ at stage~$t$. We want to let $n$ be enumerated by either $g$ or by $h$, as follows:
	If
	\[ k_t := \min\{j\in\IN \colon s_j[t] > n\} \]
	is even then we set $g(t):=0$ and $h(t):=n+1$
	(which means that $n$ is enumerated by $h$);
	otherwise, if $k_t$ is odd, then we set $g(t):=n+1$ and $h(t):=0$
	(which means that $n$ is enumerated by $g$).
	Furthermore, we define $s_i[t+1]$ for all $i\in\IN$ by
	\[ s_i[t+1] := \begin{cases}
		s_i[t] & \text{if } i \leq k_t, \\
		s_i[t] + t + 1 & \text{if } k_t < i.
	\end{cases} \]
\end{itemize}
This ends the description of stage $t$ and of the algorithm; we proceed with the verification. 

\smallskip

\emph{Claim 1.} For every $t\in\IN$, the sequence $(s_i[t])_i$ is increasing.

\smallskip

\emph{Proof.} We prove the claim by induction over $t$; it clearly holds for~$t=0$.
Fix any~$t\in\IN$ and assume that $(s_i[t])_i$ is increasing.
If $f(t)=0$ then $(s_i[t+1])_i$ is identical to $(s_i[t])_i$ and hence it is increasing as well. Thus assume that $f(t)>0$; then $k_t$ is defined
by the induction hypothesis and we can observe that $(s_i[t])_i$ is increasing:
\begin{itemize}
\item
For any $i<j\leq k_t$ we have
\[s_i[t+1] = s_i[t] < s_j[t] = s_j[t+1].\]
\item
For any $i\leq k_t < j$ we have
\[s_i[t+1] = s_i[t] < s_j[t] < s_j[t] + t + 1 = s_j[t+1].\]
\item
For any $k_t< i<j$ we have
\[s_i[t+1] = s_i[t] + t + 1 < s_j[t] + t + 1 = s_j[t+1].\]
\end{itemize}
This proves the claim.\hfill$\Diamond$

\smallskip

By Claim 1, for every $t\in\IN$ with $f(t)>0$, the number $k_t$ is well defined; thus it is clear that $g$ and $h$ as defined by the algorithm 
satisfy ($\ast$) for every $n\in\IN$.
It remains to prove the second condition in the statement of the lemma.

\smallskip

\emph{Claim 2.} For every $i$, the sequence $(s_i[t])_t$ is nondecreasing and eventually constant.

\smallskip

\emph{Proof:}
That $(s_i[t])_t$ is nondecreasing for every $i$ is clear by definition.
We show by induction over $i$ that it is eventually constant as well.

It is easy to see that $s_0[t]=0$ for all $t$. Thus 
consider an arbitrary number $i>0$.
By the induction hypothesis there exists a number $t_1$ such that, for all $t\geq t_1$, $s_{i-1}[t]=s_{i-1}[t_1]$.
Let~$t_2$~be large enough so that $t_2> t_1$ and
\[ f^{-1}\{1,\ldots, s_{i-1}[t_1]\} \subseteq \{0,\ldots,t_2-1\}  \]
(meaning that numbers smaller than $s_{i-1}[t_1]$ are enumerated by $f$ only in stages before~$t_2$).
Then, for every $t\geq t_2$ with $f(t)>0$
and for every $j<i$ we have
\[ s_j[t] \leq s_{i-1}[t] = s_{i-1}[t_1] \leq f(t)-1 , \]
hence, $k_t\geq i$ and consequently $s_i[t+1]=s_i[t]$.
By induction we obtain $s_i[t]=s_i[t_2]$, for all $t\geq t_2$.
Thus, $(s_i[t])_t$ is eventually constant, and Claim~2 is proven.\hfill$\Diamond$

\smallskip

Let the sequence $(S_i)_i$ be defined by $S_i:=\lim_{t\to\infty} s_i[t]$.
Due to Claim~1, $(S_i)_i$ is increasing.

\smallskip

\emph{Claim 3.} For every $i\in\IN$ and every $t\geq S_i$, $s_i[t]=S_i$.

\smallskip

\emph{Proof.} 
If this were not true then there would be some $t \geq S_i$ with $s_i[t+1]\neq s_i[t]$, hence,
with $S_i\geq s_i[t+1]=s_i[t]+t+1 \geq t+1 > S_i$, a contradiction.\hfill$\Diamond$

\smallskip

\emph{Claim 4.}
For every even $i$, 
$g^{-1}\{1,\ldots,S_i\}\subseteq \{0,\ldots,S_i-1\}$.

\smallskip

\emph{Proof.} 
Consider an even number $i$ as well as some $t\geq S_i$ with $g(t)>0$; we need to show that $g(t)>S_i$.
To see this, first observe that the assumption that $g(t)>0$ implies that we must have $f(t)=g(t)$ and that $k_t$ must be odd. Hence $k_t\neq i$.
If $k_t$ were smaller than $i$ then we would obtain $s_i[t+1] = s_i[t]+t+1 > s_i[t]=S_i$ in contradiction to Claim 3.
We conclude $i<k_t$. This implies $s_i[t]\leq f(t)-1$ by the definition of $k_t$.
As $t\geq S_i$, using Claim~3 again, we obtain 
\[S_i=s_i[t] \leq f(t)-1 <f(t) = g(t),\]
which proves Claim~4.\hfill$\Diamond$

\smallskip

\emph{Claim 5.}
For every odd $i$, 
$h^{-1}\{1,\ldots,S_i\}\subseteq \{0,\ldots,S_i-1\}$.

\smallskip

\emph{Proof.} The proof is symmetric to that of Claim 4; it is enough to interchange ``even'' and ``odd'' and to replace 
``$g$'' by  ``$h$''.\vphantom{\qedhere}\hfill$\quad\Diamond\,\square$
\end{proof}

As a corollary we immediately obtain the following uniformly effective version of Theorem~\ref{theorem:splitting}.
Note that the algorithm
even provides $\id_\IN$-good enumerations of $A$ and $B$. 
While we already showed in Theorem~\ref{theorem:regaining-sets} that 
$\id_\IN$-good computable enumerations exist 
for all regainingly approximable sets, its proof was not fully uniform.

\begin{theorem}
\label{theorem:splitting-effective}
Given an enumeration $f_C\colon\IN\to\IN$ of a set $C\subseteq\IN$ one can compute
enumerations without repetitions $f_A\colon\IN\to\IN$ of a set $A\subseteq\IN$ and $f_B\colon\IN\to\IN$ of a set $B\subseteq\IN$
such that 
\begin{enumerate}
\item
$C$ is the disjoint union of $A$ and $B$, and
\item
there exist infinitely many $t$ with 
\[A \cap \{0,\ldots,t-1\} \subseteq \mathrm{Enum}(f_A)[t]\]
and 
infinitely many $t$ with 
\[B \cap \{0,\ldots,t-1\} \subseteq \mathrm{Enum}(f_B)[t].\]
\end{enumerate}
\end{theorem}
\begin{proof}
First transform $f_C$ into an enumeration without repetitions~$\widetilde{f_C}$ as in Remark~\ref{remark:without_repetitions}; then apply the algorithm described in the proof of Lemma~\ref{lemma:splitting} to $\widetilde{f_C}$.
\end{proof}

Lemma~\ref{lemma:splitting} can also be applied to regular reals as defined by Wu~\cite{Wu2005}.
\begin{defi}[{Wu~\cite{Wu2005}}]\,
\begin{enumerate}
\item
For $n\in\IN$, a real number is called \emph{$n$-strongly computably enumerable} (\emph{$n$-strongly c.e.})
if it can be written as the sum of $n$~strongly left-computable numbers.
\item
If a real number is $n$-strongly c.e.\ for some $n\in\IN$, then it is called \emph{regular}.
\end{enumerate}
\end{defi}
Such numbers can be characterized conveniently by using the following
representation $\delta$ (in the sense of Weihrauch~\cite{Wei2000}) of non-negative real numbers:
Call $f\colon \IN\to\IN$ a \emph{$\delta$-name} of a non-negative real number~$\alpha$ if the series $\sum_k a_k$ defined by 
\[ a_k := \begin{cases}
    2^{-f(k)} & \text{if } f(k)>0, \\
    0 & \text{if } f(k)=0,
    \end{cases} 
\]
converges to $\alpha$. Note that, if some $f\colon\IN\to\IN$ is a $\delta$-name of a non-negative real number, then
$f^{-1}\{i+1\}$ must be finite for every $i\in\IN$; 
if we even have $|f^{-1}\{i+1\}| \leq n$ for some $n\in\IN$ and all $i\in\IN$, then we say that  $f$ is a \emph{$(\delta,n)$-name} of $\alpha$.
Then the following lemma is obvious.
\begin{lem}
\label{lemma:regular}
For $n\in\IN$, a non-negative real number $\alpha$ is $n$-strongly~c.e.\ if and only if there exists  a computable $(\delta,n)$-name $f$
of $\alpha$.
\end{lem}

Applying our previous results to the regular reals, we obtain the following final theorem.
\begin{theorem}\,
\begin{enumerate}
\item
Given a $\delta$-name (or a $(\delta,n)$-name) $f_\gamma$ of a non-negative real number $\gamma$,
one can compute $\delta$-names (or $(\delta,n)$-names) $f_\alpha$ of a non-negative real number $\alpha$
and $f_\beta$ of a non-negative real number $\beta$ with $\alpha+\beta=\gamma$ and such that
there exist infinitely many $t$ with 
\[ f_\alpha^{-1}\{1,\ldots,t\}\subseteq \{0,\ldots,t-1\} \]
and infinitely many $t$ with 
\[ f_\beta^{-1}\{1,\ldots,t\}\subseteq \{0,\ldots,t-1\} . \]
\item
For every $n$ and every $n$-strongly c.e.\ real $\gamma$ there exist $n$-strongly c.e.\ reals $\alpha$
and $\beta$ with $\alpha+\beta=\gamma$ that are additionally regainingly approximable.  
\item
For every regular real $\gamma$ there exist regular reals $\alpha$
and $\beta$ with $\alpha+\beta=\gamma$ that are additionally regainingly approximable.  
\end{enumerate}
\end{theorem}

\begin{proof}
Applying the algorithm in the proof of Lemma~\ref{lemma:splitting} 
to $f_\gamma$ proves the first assertion.
For the second assertion, use the first assertion and the following claim.

\smallskip

\emph{Claim.} Let $f$ be a $(\delta,n)$-name of a non-negative real number $\alpha$
such that there exist infinitely many $t$ with $f^{-1}\{1,\ldots,t\}\subseteq \{0,\ldots,t-1\}$.
Then $\alpha$ is regainingly approximable.

\smallskip

\emph{Proof.} 
As above, define
\[ a_k := \begin{cases}
    2^{-f(k)} & \text{if } f(k)>0, \\
    0 & \text{if } f(k)=0,
    \end{cases} 
\]
and $A_t:=\sum_{k<t} a_k$.
Then the sequence $(A_t)_t$ is a computable nondecreasing sequence converging to $\alpha = \sum_{k=0}^\infty a_k$.
For any of the infinitely many $t$ with $f^{-1}\{1,\ldots,t\}\subseteq \{0,\ldots,t-1\}$ we obtain
\[ \alpha - A_t 
  = \sum_{k=t}^\infty a_k
  \leq n \cdot \sum_{j=t+1}^\infty 2^{-j}
  = n \cdot 2^{-t} .   
\]
By Proposition~\ref{prop:equivalent-conditions} this implies that $\alpha$ is regainingly approximable.\hfill$\Diamond$

\smallskip

The third follows immediately from the second assertion.
\end{proof}

\section{Acknowledgments}

The authors would like to thank the anonymous referees for insightful remarks, in particular for suggesting to explore the Kolmogorov complexity of regainingly approximable sets and numbers, as well as 
for suggesting to extend the characterization of the c.e.~sets $A\subseteq\IN$ such that $2^{-A}$ is regainingly approximable to a 
characterization of arbitrary (not necessarily c.e.) sets with this property.

\bibliography{regaining}
\bibliographystyle{abbrv}

\end{document}